\documentclass[a4paper,10pt,reqno]{amsart}
\usepackage{amssymb,bbm}
\usepackage{array}
\usepackage{mathtools}
\usepackage[left=1in, right=1in, bottom=1.5in, top=1.5in]{geometry}
\usepackage{graphicx}
\newcolumntype{L}{>{\centering\arraybackslash}m{1.5cm}}
\usepackage{amsthm, calrsfs}
\usepackage[mathscr]{eucal}
\usepackage{rotating}
\usepackage{multirow,float}
\usepackage{color}
\numberwithin{equation}{section}
\newtheorem{theorem}{Theorem}
\newtheorem{lemma}{Lemma}

\theoremstyle{definition}

\newtheorem{examples}[theorem]{Example}
\theoremstyle{remark}
\usepackage[colorlinks=true]{hyperref}

\usepackage{enumerate}

\newcommand{\RN}[1]{\uppercase\expandafter{\romannumeral#1}}

\begin{document}
\title[Hartman--Wintner Growth Rates of Sublinear FDEs]
{Hartman--Wintner Growth Results for Sublinear Functional Differential Equations}

\author{John A. D. Appleby}
\address{School of Mathematical
Sciences, Dublin City University, Glasnevin, Dublin 9, Ireland}
\email{john.appleby@dcu.ie} \urladdr{webpages.dcu.ie/\textasciitilde
applebyj}

\author{Denis D. Patterson}
\address{School of Mathematical
Sciences, Dublin City University, Glasnevin, Dublin 9, Ireland}
\email{denis.patterson2@mail.dcu.ie} \urladdr{sites.google.com/a/mail.dcu.ie/denis-patterson}

\thanks{Denis Patterson is supported by the Government of Ireland Postgraduate Scholarship Scheme operated by the Irish Research Council under the project GOIPG/2013/402.} 
\date{1st December 2016}
\begin{abstract}
In this paper, we determine rates of growth to infinity of scalar autonomous nonlinear functional and Volterra differential equations. In these equations, the right-hand side is a positive continuous linear functional of a nonlinear function of the state. We assume the nonlinearity grows sublinearly at infinity, leading to subexponential growth in the solutions. Our main results show that the solutions of the functional differential equations are asymptotic to those of an auxiliary autonomous ordinary differential equation when the nonlinearity grows more slowly than a critical rate. If the nonlinearity grows more rapidly than this rate, the ODE dominates the FDE. If the nonlinearity tends to infinity at exactly this rate, the FDE and ODE grow at the same rate, modulo a constant non-unit factor. Finally, if the nonlinearity grows more slowly than the critical rate, then the ODE and FDE grow at the same rate asymptotically. We also prove a partial converse of the last result. In the case when the growth rate is slower than that of the ODE, we calculate sharp bounds on the solutions.
\end{abstract}
\maketitle 
\section{Introduction}
We investigate growth rates to infinity of solutions to nonlinear autonomous functional and Volterra differential equations of the form
\begin{equation} \label{eq.fde}
x'(t) = \int_{[-\tau,0]}\mu(ds)f(x(t+s)), \quad t >0;\quad\quad x_0 = \psi \in C([-\tau,0];(0,\infty)), 
\end{equation}
and 
\begin{equation} \label{eq.vde}
x'(t)=\int_{[0,t]} \mu(ds)f(x(t-s)), \quad t\geq 0;\quad x(0)=\psi>0.
\end{equation}
Concentrating momentarily on \eqref{eq.fde}; we suppose that $\tau>0$ and $\mu$ is a positive finite Borel measure on $[-\tau,0]$, so $\mu(E)\in [0,\infty)$ for all Borel sets $E\subseteq[-\tau,0]$ and $\mu([-\tau,0])=:M\in (0,\infty)$. In the case of \eqref{eq.vde}, we have $M:=\mu([0,\infty))$. If $f$ is positive, by the Riesz representation theorem, \eqref{eq.fde} is equivalent to $x'(t)=L([f(x)]_t)$ for $t>0$, where $L$ is a positive continuous linear functional from $C([-\tau,0];\mathbb{R}^+)$ to $\mathbb{R}^+$. Uniqueness of a continuous solution of \eqref{eq.fde} or \eqref{eq.vde} is guaranteed by asking that $f$ is continuously differentiable (see \cite{gripenberg1990volterra} for existence results and properties of measures); positivity of solutions is guaranteed by the positivity of $\mu$ and $f$ on $[0,\infty)$. Non--explosion of solutions in finite time, as well as subexponential growth to infinity (in the sense that $\log x(t)/t\to 0$ as $t\to\infty$), follows from the hypothesis that $f'(x)\to 0$ as $x\to\infty$.

When $f$ is a positive continuous function such that 
\begin{equation} \label{fasym}
\text{there exists $\phi\in \mathcal{S}$ such that $f(x) \sim \phi(x)$ as $x\to\infty$}
\end{equation}
where $\mathcal{S}$ is the class 
\begin{align} \label{def.smoothsublinear}
\mathcal{S}=\{\,\, &\phi\in C^1((0,\infty);(0,\infty))\cap C(\mathbb{R}^+,(0,\infty)):
\text{$\lim_{x\to\infty}\phi'(x)=0$ and $\phi'(x) > 0$ for all $x>0$} \,\,\},
\end{align}
then 
\[
\lim_{t\to\infty} \frac{F(x(t))}{t}=M,
\]
where 
\begin{equation} \label{def.F}
F(x) = \int_1^x \frac{1}{f(u)}du, \quad x>0
\end{equation}
(see \cite{applebypatterson2016growth} for further details). Furthermore,  
\begin{equation} \label{eq.RVHWcondn}
\limsup_{x\to\infty} \frac{f(x)F(x)}{x}<+\infty
\end{equation}
implies 
\[
\lim_{t\to\infty} \frac{x(t)}{F^{-1}(Mt)}=1.
\]

The theorems stated above develop results in \cite{subexpRV} which require coefficients to be regularly varying at infinity, and consider only a single fixed delay. Since we refer often to the class of regularly varying function, we remind the reader of the definition (see \cite{BGT}): a measurable function $g:(0,\infty)\to (0,\infty)$ is regularly varying at infinity with index $\beta\in \mathbb{R}$ if $g(\lambda t)/g(t)\to \lambda^\beta$ as $t\to\infty$, for every $\lambda>0$, and we write $g\in \text{RV}_\infty(\beta)$. 

Therefore, under \eqref{eq.RVHWcondn}, the rates of growth of solutions of \eqref{eq.fde} and of 
\begin{equation} \label{eq.equivode}
y'(t)=Mf(y(t)), \quad t>0; \quad y(0)=y_0>0
\end{equation}
are the same, in the sense that $x(t)/y(t)\to 1$ as $t\to\infty$. The non--delay equation \eqref{eq.equivode} can be considered as a special type of equation \eqref{eq.fde} in which all the mass of $\mu$ is concentrated at $0$. On the other hand, if $f$ is linear, collapsing the mass of $\mu$ to zero generates different rates of (exponential) growth in the solutions of \eqref{eq.fde} and \eqref{eq.equivode}. The condition \eqref{eq.RVHWcondn} holds for $f\in \text{RV}_\infty(\beta)$ where $\beta<1$, but does not hold if $f$ is in $\text{RV}_\infty(1)$. Therefore, the phenomenon that solutions of \eqref{eq.equivode} yield the growth rate of those of \eqref{eq.fde} ceases for some critical rate of growth of $f$ faster than functions in $\text{RV}_{\infty}(\beta)$ for $\beta<1$, but slower than linear. 

In \cite{subexp}, the authors showed (under some technical conditions) that the critical growth rate is $\text{O}(x/\log x)$: more precisely, if we define     
\begin{equation} \label{eq.lambda}
\lambda := \lim_{x \to \infty}\frac{f(x)}{x/\log(x)} \in [0,\infty],
\end{equation}
and $C:=\int_{[-\tau,0]} |s|\mu(ds)$, then 
\begin{equation} \label{eq.hwresult}
\lim_{t\to\infty} \frac{x(t)}{y(t)} =e^{-\lambda C},
\end{equation}
provided $f$ is ultimately increasing and $f'\in \text{RV}_\infty(0)$, a hypothesis stronger than, but implying $f\in \text{RV}_\infty(1)$. In this paper one of our main results (Theorem~\ref{thm.2.2}) extends the results from \cite{subexp} by removing entirely the assumption that $f'\in \text{RV}_\infty(0)$: instead, we assume that $f\in \mathcal{S}$ (with $\mathcal{S}$ as in \eqref{def.smoothsublinear}). As mentioned above 
\begin{equation} \label{eq.hgresult}
\lim_{t\to\infty} \frac{F(x(t))}{t}= M,  \quad \lim_{t\to\infty}\frac{F(y(t))}{t}=M.
\end{equation}
In the linear case, the asymptotic relation \eqref{eq.hgresult} would mean that $x$ and $y$ share the same Liapunov exponent, but would not necessarily obey $x(t)\sim Ky(t)$ as $t\to\infty$. Therefore our results identify a subtle distinction in the growth rates of $x$ and $y$, which are in some sense closer than Hartman--Grobman type of asymptotic equivalence embodied by \eqref{eq.hgresult}. By 
contrast, the relation \eqref{eq.hwresult} is in the spirit of 
a Hartman--Wintner type--result (see \cite[Cor X.16.4]{H}, \cite{HW}). 
We note of course, that there is a huge literature in asymptotic integration and Hartman--Wintner type--results in determining the asymptotic behaviour of functional differential equations (see e.g., \cite{arino1983asymptotic,arino1989asymptotic, castillo2004asymptotics,graef,haddock1980stability,KO,pinto1993asymptotic,pituk} and the introductions of \cite{castillo2004asymptotics,nesterov2013asymptotic} for reviews of the development of the literature to date). However, most work in the literature is concerned with equations whose leading order behaviour is linear, with perturbed terms either being nonautonomous, or of smaller than linear order. 
In our work, as $f(x)/x\to 0$ as $x\to\infty$, no leading order linear behaviour is present, necessitating a different approach.  

When $\lambda=+\infty$, equation \eqref{eq.hwresult} reads $x(t)=o(y(t))$ as $t\to\infty$. However, we are still able to determine the rate of growth relatively precisely in this case, under the additional assumption that $f'$ is decreasing. In Theorem~\ref{thm.2.2.1.1} we show that
\[
x(t)=F^{-1}(Mt - c(t)\log F^{-1}(Mt)), \quad t\geq 1,
\]
where $c$ is a $C^1$ function such that $c(t)\to C$ as $t\to\infty$. 

We also prove results for the Volterra differential equation \eqref{eq.vde} 
where $\mu\in M([0,\infty);\mathbb{R}^+)$. In this case, with $\lambda$ defined by \eqref{eq.lambda}, we obtain 
\begin{equation} \label{eq.hwresult2}
\lim_{t\to\infty} \frac{x(t)}{y(t)} =\exp\left(-\lambda\int_{[0,\infty)} s\mu(ds)\right),
\end{equation}
except possibly in the case when $\lambda=0$ and 
\[
\int_{[0,\infty)} s\mu(ds)=+\infty
\]
(see Theorem~\ref{thm.2.3}). In this last case, we provide necessary and sufficient conditions under which $x(t)/y(t)\to 1$ or $x(t)/y(t)\to 0$ as $t\to\infty$ (Theorem~\ref{thm.2.4}). We do not believe that the sufficient conditions given in Theorem~\ref{thm.2.4} are sharp in general. Hence, when $f$ is regularly varying with unit index at infinity and $C=+\infty$, we provide what we believe is a sharp necessary condition under which $x(t)/y(t)\to 1$ as $t\to\infty$ in Theorem~\ref{thm.2.5}.

For both \eqref{eq.fde} and \eqref{eq.vde}, in the case when the first moment of the measure $\mu$ is finite, we show that the critical growth rate $f(x)=o(x/\log x)$ as $x\to\infty$ is a sharp condition to obtain $x(t)/y(t)\to 1$ as $t\to\infty$. More
precisely in Theorem~\ref{thm.2.2.1} we see that when $f'$ is decreasing, then $f(x)=o(x/\log x)$ as $x\to\infty$ and $x(t)/y(t)\to 1$ as $t\to\infty$ are equivalent.

The structure of the paper is as follows: in Section~\ref{sec.mainresultsHW}, we state and discuss the main results of the paper. Section~\ref{sec.HWexamples} contains examples. An important lemma which allows direct asymptotic information about the solution to be deduced is given in Section~\ref{sec.asylemma}. The remaining sections of the paper are devoted to the proofs of the main results.
   

\section{Main Results} \label{sec.mainresultsHW}
In what follows, we interpret 
\[
e^{-\infty}:=0
\]
in order to streamline the statement of results. 
We first state our main result for the solution of the functional differential equation \eqref{eq.fde}.
\begin{theorem}  \label{thm.2.2}
Let $f(x)>0$ for all $x>0$, $f'(x)>0$ for all $x > x_1$, $f'(x) \to 0$ as $x \to \infty$. 
Suppose $f$ obeys \eqref{eq.lambda}, let $\tau>0$, $\mu\in M([-\tau,0];\mathbb{R}^+)$ be a positive finite Borel measure,
with 
\[
M := \int_{[-\tau,0]} \mu(ds), \quad C := \int_{[-\tau,0]} |s|\mu(ds), 
\] 
$F$ is defined by \eqref{def.F}, and $x$ is the unique continuous solution $x$ of  \eqref{eq.fde}.
Then 
\[
\lim_{t\to\infty} x(t)=+\infty, \quad
\lim_{t\to\infty} \frac{F(x(t))}{t}=M,
\]
and moreover 
\begin{equation} \label{eq.xasy}
\lim_{t \to \infty}\frac{x(t)}{F^{-1}(Mt)} = e^{-\lambda C}.
\end{equation}
\end{theorem}
The proof of this result, and others like it, consists of two main steps. The first step is to show that $x$ obeys 
\begin{equation}  \label{eq.keylimit1}
\lim_{t\to\infty} \frac{F(x(t))-Mt}{\log f(x(t))}=-C.
\end{equation}
Equation \eqref{eq.keylimit1} is also true for solutions of the Volterra equation \eqref{eq.vde}, even when the first moment of the measure in that case is infinite. A key step in proving \eqref{eq.keylimit1} is to rewrite \eqref{eq.fde} in the form  
\begin{equation*} 
x'(t)=Mf(x(t))-\int_{[-\tau,0]} \mu(ds) \{f(x(t))-f(x(t+s))\}=:Mf(x(t))-\delta(t),
\end{equation*}
thereby viewing \eqref{eq.fde} as a perturbation of \eqref{eq.equivode}. Clearly, if the perturbed term $\delta$ (which will be positive for large $t$, by the monotonicity of $x$ and $f$) is small relative to $Mf(x(t))$, we may expect $x(t)/y(t)$ to tend to a finite limit. The first main task is therefore to determine precise asymptotic information on $\delta$. 

Remarkably, in spite of the path dependence of $x$ in $\delta$, we show that $\delta(t)\sim -C\log f(x(t))$ as $t\to\infty$, and from this \eqref{eq.keylimit1} readily follows. The second step in the proof of Theorem~\ref{thm.2.2} can be found in  Lemma~\ref{lemma.l1} and involves viewing the limit in \eqref{eq.keylimit1} as a pair of asymptotic inequalities, from which the implicit asymptotic information about $x$ can be made explicit, as in \eqref{eq.xasy}.   

We note that under these hypotheses we have $f(x)/x \to 0$ as $x \to \infty$. Since $f$ is ultimately increasing it must either have a finite limit or tend to infinity as $x \to \infty$. In the former case, $x'(t)$ tends to a finite limit, and \eqref{eq.xasy} is trivially true. Hence we assume, without loss of generality, in all the results and proofs below that  $f(x)\to\infty$ as $x\to\infty$.

We may take $C>0$ in Theorem~\ref{thm.2.2}: the finiteness of the measure automatically ensures that $C$ is finite.
If $C=0$, it must follow that $\mu(ds)=M\delta_{0}(ds)$ a.e. and so \eqref{eq.fde} collapses to the ODE 
\eqref{eq.equivode}, rendering the result trivial. Therefore, it is tacit in this result, and in subsequent theorems for Volterra equations, that the first moment of $\mu$, $C$, is positive. With this in mind, we now see that 
the solution of \eqref{eq.fde} is exactly asymptotic to the solution of \eqref{eq.equivode} when $\lambda=0$, because in this case
\[
\lim_{t\to\infty} \frac{x(t)}{F^{-1}(Mt)}=1.
\]
However, a non--unit limit exists once $\lambda$ is positive or infinite. 

When $\lambda=+\infty$, and $C>0$, we should interpret \eqref{eq.xasy} as 
\[
\lim_{t\to\infty} \frac{x(t)}{F^{-1}(Mt)}=0.
\]
This leads us to ask: can we still get \emph{direct} asymptotic information about the slower rate of growth of $x$ in this case? The next result shows that we can, at the cost of assuming $f'$ is decreasing. 
\begin{theorem} \label{thm.2.2.1.1}
Let $f(x)>0$ for all $x>0$, $f'(x)>0$ for all $x > x_1$, $f'(x) \to 0$ as $x \to \infty$. Suppose $f$ obeys \eqref{eq.lambda},
with $\lambda=+\infty$, and $f'$ is decreasing on $[x_2,\infty)$. Let $\tau>0$, $\mu\in M([-\tau,0];\mathbb{R}^+)$ be a positive finite Borel measure, with 
\[
M := \int_{[-\tau,0]} \mu(ds), \quad C := \int_{[-\tau,0]} |s|\mu(ds)<+\infty, 
\] 
$F$ is defined by \eqref{def.F}, and $x$ is the unique continuous solution $x$ of  \eqref{eq.fde}. Then 
there is a $c\in C^1((1,\infty);\mathbb{R})$ with $\lim_{t\to\infty} c(t)=C$ such that 
\begin{equation}   \label{eq.xasycorrection}
x(t)=F^{-1}\left(Mt - c(t)\log F^{-1}(Mt)\right), \quad t\geq 1.
\end{equation}
\end{theorem}

The assumption that $f'$ is decreasing is used to show that $\log f(x(t))\sim \log f(F^{-1}(Mt))$ as $t\to\infty$ 
(using Lemma~\ref{asym_equiv_copy_ap}). 
Once this is achieved, \eqref{eq.keylimit1} immediately gives  
\[
\lim_{t\to\infty} -\frac{F(x(t))-Mt}{\log F^{-1}(Mt)}=C,
\]
because $\log f(x)/\log x\to 1$ as $x\to\infty$ when $\lambda=+\infty$. Defining $c$ to be the function in the last limit
now gives \eqref{eq.xasycorrection}. This approach could be used to prove \textit{all} cases in Theorem~\ref{thm.2.2} directly, rather than by appealing to the implicit arguments used in Lemma~\ref{lemma.l1} (i.e., in the second step of the proof of Theorem~\ref{thm.2.2}). The direct argument would then proceed by means of Lemma~\ref{lemma.lemmaa1} and related results. 

Given the asymptotic taxonomy established in Theorem~\ref{thm.2.2}, one might ask whether the condition that $f(x)/(x/\log x)\to 0$ as $x\to\infty$ is \emph{necessary} in order to preserve the asymptotic behaviour of \eqref{eq.equivode}. The next result shows that it is. 
\begin{theorem} \label{thm.2.2.1}
Let $f(x)>0$ for all $x>0$, $f'(x)>0$ for all $x > x_1$, $f'(x) \to 0$ as $x \to \infty$. 
Suppose in addition $f'$ is decreasing on $[x_2,\infty)$. 
Let $\tau>0$, $\mu\in M([-\tau,0];\mathbb{R}^+)$ be a positive finite Borel measure,
with 
\[
M := \int_{[-\tau,0]} \mu(ds), \quad C := \int_{[-\tau,0]} |s|\mu(ds), 
\] 
$F$ is defined by \eqref{def.F}, and $x$ is the unique continuous solution $x$ of  \eqref{eq.fde}.
Then the following are equivalent:
\begin{itemize}
\item[(a)]
\[
\lim_{x\to\infty} \frac{f(x)}{x/\log x} = 0;
\]
\item[(b)]
\begin{equation*} 
\lim_{t \to \infty}\frac{x(t)}{F^{-1}(Mt)} = 1.
\end{equation*}
\end{itemize}
\end{theorem}
The extra hypothesis that $f'$ is monotone is needed to prove that (b) implies (a): the proof that (a) implies (b)
can still be established using the hypotheses of Theorem~\ref{thm.2.2}. 


We now state the result analogous to Theorem~\ref{thm.2.2} for the solution of the Volterra differential equation \eqref{eq.vde}.
\begin{theorem}  \label{thm.2.3}
Let $f(x)>0$ for all $x>0$, $f'(x)>0$ for all $x > x_1$, $f'(x) \to 0$ as $x \to \infty$. 
Suppose $f$ obeys \eqref{eq.lambda}, $\mu\in M([0,\infty);\mathbb{R}^+)$ is a positive finite Borel measure,
with 
\[
M := \int_{[0,\infty)} \mu(ds), \quad C := \int_{[0,\infty)} s\mu(ds),
\] 
$F$ is defined by \eqref{def.F}, and $x$ is the unique continuous solution $x$ of \eqref{eq.vde}.
\begin{itemize}
\item[(a)] $x$ obeys
\[
\lim_{t\to\infty} x(t)=+\infty, \quad
\lim_{t\to\infty} \frac{F(x(t))}{t}=M.
\]
\item[(b)] If $C<+\infty$, then  
\begin{equation} \label{eq.xasyv}
\lim_{t \to \infty}\frac{x(t)}{F^{-1}(Mt)} = e^{-\lambda C}.
\end{equation}
\item[(c)] If $C=+\infty$ and $\lambda\in (0,\infty]$ then  \eqref{eq.xasyv} still prevails.
\end{itemize}
\end{theorem}
In the case when $C$ is finite, we can prove a result for \eqref{eq.vde} exactly analogous to Theorem~\ref{thm.2.2.1} for \eqref{eq.fde}, namely that $x(t)/F^{-1}(Mt)\to 1$ if and only if $f(x)\log x/x\to 0$ 
as $x\to\infty$, under the additional assumption that $f'(x)$ tends to zero monotonically. Moreover, we also have a result 
for \eqref{eq.vde} which is an exact analogue of  Theorem~\ref{thm.2.2.1.1} for \eqref{eq.fde}, again assuming $f'(x)$ tends to zero monotonically.

In the functional differential equation \eqref{eq.fde}, $C$ is always finite. However, if $\mu$ is a non--negative 
nontrivial finite measure in $M([0,\infty);\mathbb{R}^+)$, the first moment $C$ can be infinite. In this situation, if $\lambda\in (0,\infty)$, it can now happen that 
\[
\lim_{t \to \infty}\frac{x(t)}{F^{-1}(Mt)} = 0,
\] 
which is in contrast to the finite memory case. Of course, if $\lambda=+\infty$, it does not matter whether 
$C$ is finite or not, and we have 
\[
\lim_{t \to \infty}\frac{x(t)}{F^{-1}(Mt)} = 0,
\]
which is the same as we see in the finite memory case.

It can therefore be seen that Theorem~\ref{thm.2.3} addresses all cases except for that when $\lambda=0$, $C=\infty$. 
Again, the different effect that unbounded memory can have on the asymptotic behaviour is demonstrated: for \eqref{eq.fde}, if $\lambda=0$, it must follow that 
\[
\lim_{t\to\infty} \frac{x(t)}{F^{-1}(Mt)}=1.
\]  
However, this is not guaranteed to be the case for solutions of \eqref{eq.vde}. The condition 
\begin{equation} \label{eq.suffcrude}
\limsup_{x\to\infty} \frac{f(x)}{x}\int_1^x \frac{1}{f(u)}\,du <+\infty
\end{equation}
is nevertheless sufficient to ensure the existence of a unit limit in \eqref{eq.xasyv}, and roughly speaking, this condition is true for functions which grow more slowly that $x^{1-\epsilon}$ for some $\epsilon\in (0,1)$ 
(more precisely it is true, if $f\in \text{RV}_\infty(1-\epsilon)$ for some $\epsilon\in (0,1)$ or if $x\mapsto f(x)/x^{1-\epsilon}$ is asymptotic to a decreasing function) \cite{applebypatterson2016growth}. In the case that $f(x)/x\to 0$ as $x\to\infty$, and 
$f$ in $\text{RV}_\infty(1)$, it is true that 
\begin{equation} \label{eq.suffcrudefalse}
\lim_{x\to\infty} \frac{f(x)}{x}\int_1^x \frac{1}{f(u)}\,du =+\infty,
\end{equation}
so the potential arises for a limit less than unity in \eqref{eq.xasyv} even when  
\[
\lim_{x\to\infty} \frac{f(x)}{x/\log x}=0
\] 
and $C=+\infty$.

Our last result shows that different limits can indeed result in the case when $\lambda=0$, $C=\infty$, depending on how slowly $\int_0^t \int_{[s,\infty)} \mu(du) \,ds \to\infty$ as $t\to\infty$. We do not give a classification in all cases, but merely give sufficient conditions for  the limit in \eqref{eq.xasyv} to be zero or unity, and briefly show that some of our sufficient conditions are also sometimes necessary. In order to simplify proofs, we assume here that $f$ is increasing on $[0,\infty)$.

\begin{theorem}  \label{thm.2.4}
Let $f(x)>0$ for all $x>0$, $f'(x)>0$ for all $x > 0$, $f'(x) \to 0$ as $x \to \infty$. Suppose $f$ obeys \eqref{eq.lambda}, and $\mu\in M([0,\infty);\mathbb{R}^+)$ is a positive finite Borel measure, $F$ is given by \eqref{def.F}, $M := \int_{[0,\infty)} \mu(ds)$ and let $x$ be the unique continuous solution $x$ of \eqref{eq.vde}. 
\begin{itemize}
\item[(i)] If
\begin{align} \label{eq.suff11}
\lim_{x\to\infty} \frac{f(x)}{x/\log x}\int_{[0,F(x)/M]} s\mu(ds)&=0, \\
\label{eq.suff21}
\lim_{x\to\infty} \frac{f(x)}{x}\int_{[0,F(x)/M]} \int_{[s,\infty)}  \mu(du)\,ds&=0,
\end{align}
then
\begin{equation} \label{eq.xasyv2}
\lim_{t \to \infty}\frac{x(t)}{F^{-1}(Mt)} = 1.
\end{equation}
\item[(ii)]
If 
\begin{equation} \label{eq.suff10}
\lim_{x\to\infty} \frac{f(x)}{x}\int_{[0,F(x)/M]} \int_{[s,\infty)}  \mu(du)\,ds=+\infty,
\end{equation}
then 
\begin{equation} \label{eq.xasyv3}
\lim_{t \to \infty}\frac{x(t)}{F^{-1}(Mt)} = 0.
\end{equation}
\item[(iii)] If $f'$ is decreasing on $[x_2,\infty)$, then \eqref{eq.xasyv2} implies \eqref{eq.suff21}. 
\end{itemize}
\end{theorem}

We note that the condition \eqref{eq.suff21} is a consequence of the condition \eqref{eq.suffcrude}, and if  
\eqref{eq.suff10} holds, then \eqref{eq.suffcrude} cannot: indeed \eqref{eq.suff10} implies \eqref{eq.suffcrudefalse}.

We give some examples in the next section which illuminate the sufficient conditions 
\eqref{eq.suff11}, \eqref{eq.suff21}, \eqref{eq.suff10} under which we obtain unit or zero limits. 
However, it can be seen that if the rate of growth of 
\[
t\mapsto \int_{[0,t]} \int_{[s,\infty)}  \mu(du)\,ds=:T(t)
\] 
to infinity as $t\to\infty$ is faster, it is more likely that the solution of \eqref{eq.vde} will grow strictly more 
slowly than that of \eqref{eq.equivode}, and the slower that $T$ grows, and the faster that 
\[
x\mapsto \frac{f(x)}{x/\log x}
\]
tends to zero as $x\to\infty$, the more likely it is that the solution of \eqref{eq.vde} will inherit exactly the 
rate of growth of the solution of \eqref{eq.equivode}. 

We do not attempt to improve the sufficient conditions in Theorem~\ref{thm.2.4} here. 
As the discussion above suggests, when $f$ grows more slowly than a function in $\text{RV}_\infty(1)$, a unit limit in 
\eqref{eq.xasyv} is usually admitted. However, when $f$ is in $\text{RV}_\infty(1)$ with $\lambda=0$, it is interesting to speculate how close 
\eqref{eq.suff11} is to being necessary in order to obtain a unit limit in \eqref{eq.xasyv} (part (iii) confirms that \eqref{eq.suff21} is necessary if $f$ is ultimately concave). 
\begin{theorem} \label{thm.2.5}
Let $f'(x)>0$ for all $x>0$ and $f'(x)\to 0$ as $x\to\infty$ with $f'$ decreasing. Suppose that $f\in \text{RV}_\infty(1)$ such that 
\[
\lim_{x\to\infty} \frac{xf'(x)}{f(x)}=1.
\] 
Let $\mu\in M([0,\infty);\mathbb{R}^+)$ be a positive finite Borel measure, $F$ is given by \eqref{def.F}, $M := \int_{[0,\infty)} \mu(ds)$ and let $x$ be the unique continuous solution $x$ of \eqref{eq.vde}.
Define 
\begin{equation} \label{def.K}
K(x) = \int_1^x \left\{\frac{f(v)}{v}\int_{[F(x)/M-F(v)/M,F(x)/M]} \mu(ds)\right\}  \,dv.
\end{equation}
If $x$ obeys \eqref{eq.xasyv2}, then \eqref{eq.suff21} and  
\begin{equation} \label{eq.suff31}
\lim_{x\to\infty} \frac{f(x)}{x} \int_{1}^x  K(u) \frac{1}{f^2(u)}\,du = 0,
\end{equation}
hold. 
\end{theorem}
We have not made extensive use of the theory of regular variation in this paper, even in Theorem~\ref{thm.2.5}. However, it seems that extracting good asymptotic information along the lines needed to prove a converse of Theorem~\ref{thm.2.5} may make greater requests on this theory. The literature regarding the application of the theory of regular variation to the 
asymptotic behaviour of ordinary and functional differential equations is extensive and growing (see for example the monographs of Mari\'{c} \cite{maric2000regular} and {\v{R}}eh{\'a}k \cite{rehakrv} and recent representative papers such as \cite{chatzarakis2014precise}, \cite{matucci2014asymptotics}, \cite{matucci2015extremal} and \cite{takasi2011precise}).

\section{Examples} \label{sec.HWexamples}
\begin{examples} \label{example.1}
A simple example of a function $f$ which obeys the hypotheses of all theorems is now given. We use it throughout this section to illustrate the scope of our general results. Let 
$g(x) = (x+1)/\log^\theta(2+x)$, for $\theta>0$. Clearly $g(x)>0$ for $x>0$ and 
\[
g'(x) = \frac{1}{\log^\theta(2+x)}\left(1 - \frac{(1+x)\theta}{(2+x)\log(2+x)} \right) > 0,\,\, x > e^\theta-2=:s_1(\theta)>0.
\] 
It is easy to see that $g'(x) \to 0$ as $x \to \infty$. 
Moreover, 
\[
g''(x) = \frac{\theta \log^{-(\theta+2)}(x+2) \{(\theta+1) (x+1)-(x+3) \log(x+2)\}}{(x+2)^2}.
\]
Since $x+3>x+1$, by considering the term in the curly brackets, we have $g''(x)<0$ for all $x>e^{\theta+1}-2=:s(\theta)>s_1(\theta)$.
Now, define $f(x)=g(x+s(\theta))$ for $x\geq 0$. Then by the definition of $g$, we see that $f(x)>0$ for all $x\geq 0$, 
$f'(x)>0$ for all $x>0$ and $f''(x)<0$ for all $x>0$. This function $f$ fulfills the hypotheses of all main results, but notice that 
taking $f=g$ still suffices for all results in which we only require $f'(x)>0$ for $x$ sufficiently large. 

By construction, $\lambda$ in \eqref{eq.lambda} is 0, 1, or $+\infty$ according to whether $\theta$ is greater than, equal to, or less than, unity. Computing $F$ simply involves making a substitution and splitting the resulting integral; doing so yields the formula 
\begin{align*}
F(x) = \frac{1}{1+\theta}\log^{\theta+1}\left(x+ e^{\theta +1}\right) - \frac{1}{1+\theta}\log^{\theta+1}\left( 1+ e^{\theta + 1}\right) + \int_{\log \left( 1+ e^{\theta + 1} \right)}^{\log \left(x+ e^{\theta + 1}\right)} \frac{w^\theta}{e^w-1}dw, \quad x > 1.
\end{align*}
From here it is straightforward to show that 
\[
F(x) \sim \frac{1}{1+\theta}\log^{\theta+1}(x), \,\, F^{-1}(x) \sim \exp\left((\theta+1)^{\frac{1}{\theta+1}}x^{\frac{1}{\theta+1}}\right), \mbox{ as } x \to \infty. 
\]
Using the notation for $M$ and $C$ in Theorem \ref{thm.2.2}, the solution of \eqref{eq.fde} obeys
\[
x(t)\sim 
\left\{
\begin{array}{cc}
o\left(\exp\left((\theta+1)^{\frac{1}{\theta+1}}(Mt)^{\frac{1}{\theta+1}}\right)\right),  &\theta <1, \\
e^{-C} \exp\left((\theta+1)^{\frac{1}{\theta+1}}(Mt)^{\frac{1}{\theta+1}}\right),  &\theta =1, \\
\exp\left((\theta+1)^{\frac{1}{\theta+1}}(Mt)^{\frac{1}{\theta+1}}\right),  &\theta >1, 
\end{array}
\right.
\]
as $t\to\infty$. Naturally, one can obtain the same asymptotic representation for the solution of \eqref{eq.vde} by Theorem \ref{thm.2.3} in the case where 
$C=\int_{[0,\infty)} s\mu(ds)$ is finite.
\end{examples}

\begin{examples} \label{example.2}
In this example, we show, in many cases of interest, that \eqref{eq.suff11} implies \eqref{eq.suff21}.
We can see, roughly, that a claim of this type would follow from information about the relative 
asymptotic behaviour of 
\[
t\mapsto \int_{[0,t]} u\mu(du) \quad \text{  and  } \quad t\mapsto t\int_{[t,\infty)} \mu(du) \text{ as $t\to\infty$}
\] 
because, for any $t\geq 0$, we have 
\begin{equation}  \label{eq.mumomemttail}
\int_0^t \int_{[s,\infty)} \mu(du)\,ds = \int_{[0,t]} u\mu(du) + t\int_{[t,\infty)} \mu(du).
\end{equation}

We specialise to the case when $\mu\in M([0,\infty);\mathbb{R}^+)$ is absolutely continuous and therefore 
we have $\mu(ds)=k(s)\,ds$ where $k$ is continuous, non--negative and integrable. Hence 
for every Borel set $E\subset [0,\infty)$ we have 
\[
\mu(E)=\int_E k(s)\,ds.
\] 
Now suppose further that $k\in \text{RV}_\infty(-\alpha)$. Then integrability forces $\alpha\geq 1$. Also, if 
$\alpha>2$, it follows that 
\[
C= \int_{[0,\infty)} s\mu(ds) = 
\int_{0}^\infty \int_{[t,\infty)} \mu(ds) \,dt <+\infty,
\]
so to be of interest in Theorem~\ref{thm.2.4}, it is necessary for $\alpha\in [1,2]$. 

In the case $\alpha\in (1,2)$, we have by Karamata's theorem (see e.g.~\cite[Theorem 1.5.11]{BGT})
\[
t
\int_{[t,\infty)} \mu(ds) \sim \frac{1}{\alpha-1} t^2k(t), \quad
\int_{[0,t]} s\mu(ds) \sim \frac{1}{2-\alpha} t^2 k(t), \quad \text{ as $t\to\infty$}.
\]
Hence by \eqref{eq.mumomemttail},
\begin{equation}  \label{eq.momenttailasy}
\int_0^t \int_{[s,\infty)} \mu(du) \,ds \sim \left( 1 + \frac{2-\alpha}{\alpha-1} \right) \int_{[0,t]} s\mu(ds), 
\quad \text{ as $t\to\infty$}.
\end{equation}
Therefore, for $\alpha\in (1,2)$, if \eqref{eq.suff11} holds, then so does \eqref{eq.suff21}. Karamata's theorem 
applied to $t\mapsto \int_{[0,t]} s\mu(ds)$ also shows that this implication is true if $\alpha=2$ and $C=+\infty$. 

%
\end{examples}

\begin{examples} \label{example.3}
Let $f$ be as in Example \ref{example.1}. Suppose that $\theta>1$ and note that $f\in \text{RV}_\infty(1)$, so 
\[
\lim_{x\to\infty} \frac{f(x)F(x)}{x}=\infty
\]
 and $\lambda=0$ in \eqref{eq.lambda}. Therefore, in order to check whether $x(t)/F^{-1}(Mt)$ tends to a non--unit limit, it is necessary to appeal to Theorem~\ref{thm.2.4} in the case when $C=+\infty$.
We saw in Example \ref{example.2} that choosing $\mu$ to be absolutely continuous with $\mu(ds)=k(s)\,ds$ and 
$k\in \text{RV}_\infty(-\alpha)$ for $\alpha\in [1,2]$ allows us to consider the case when $C=+\infty$. Therefore, let 
 $k\in \text{RV}_\infty(-\alpha)$ for $\alpha\in [1,2]$.

We now show, using Theorem~\ref{thm.2.4}, that 
\begin{equation} \label{eq.lim1}
\alpha\in \left( 1 + \frac{2}{1+\theta},2\right] \text{ implies } \lim_{t\to\infty} \frac{x(t)}{F^{-1}(Mt)}=1
\end{equation}
while 
\begin{equation} \label{eq.lim2}
\alpha\in \left[ 1, 1+\frac{1}{1+\theta}\right) \text{ implies } \lim_{t\to\infty} \frac{x(t)}{F^{-1}(Mt)}=0
\end{equation}
in the case that $k\in L^1(0,\infty)$.

Therefore, the slower that $f$ grows, the larger is $\theta$, and 
the greater the range of $\alpha$ for which \eqref{eq.lim1} holds: hence, less rapid growth in $f$ makes it easier for the asymptotic behaviour of \eqref{eq.equivode} to be preserved by the solution of \eqref{eq.vde}. On the other hand, as $\theta\downarrow 1$, the range of values of  $\alpha$ for which \eqref{eq.lim1} holds
narrows, and indeed collapses to the singleton $\alpha\in \{2\}$. 

Viewing $\theta$ as fixed, we see that the larger the value of $\alpha$, and the more rapidly the memory of the past fades, the more likely it is that  \eqref{eq.lim1} holds, and the asymptotic behaviour of \eqref{eq.equivode} to be preserved by the solution of \eqref{eq.vde}. Turning to \eqref{eq.lim2}, 
similar considerations connect the relative strength of the nonlinearity and the 
rapidity at which the memory fades, leading to growth in $x$ which is slower than that in the solution of \eqref{eq.equivode}.

We prove the claims \eqref{eq.lim1} and \eqref{eq.lim2}. 
With $F$ defined by \eqref{def.F}, we have 
\begin{equation}  \label{eq.fFasymptoticsex}
F(x)\sim \frac{1}{\theta+1} (\log x)^{1+\theta}, \quad 
\frac{f(x)}{x/\log x} \sim (\log x)^{1-\theta} 
, \quad \text{as $x\to\infty$}.
\end{equation}
By Karamata's theorem,
\begin{equation} \label{eq.intsareRV}
t\mapsto 
\int_{[0,t]} s\mu(ds) \in \text{RV}_\infty(2-\alpha), \quad 
t\mapsto \int_{[t,\infty)} \mu(ds) \in \text{RV}_\infty(1-\alpha).
\end{equation}
Hence by \eqref{eq.intsareRV} and \eqref{eq.fFasymptoticsex}, as $x\to\infty$,
\[
\int_{[0,F(x)/M]} s\mu(ds)\sim \int_{0}^{\frac{1}{M(\theta+1)}\log^{1+\theta} x} sk(s)\,ds
\sim \left(\frac{1}{M(\theta+1)}\right)^{2-\alpha} \int_{0}^{\log^{1+\theta} x} sk(s)\,ds,
\]
so \eqref{eq.suff11} is equivalent to 
\[
\lim_{x\to\infty} (\log x)^{1-\theta}\int_{0}^{\log^{1+\theta} x} sk(s)\,ds=0.
\]
This in turn is equivalent to 
\begin{equation} \label{eq.firstone}
\lim_{t\to\infty} t^{\frac{1-\theta}{1+\theta}}\int_0^t sk(s)\,ds=0.
\end{equation}
Therefore, by the last example and Theorem~\ref{thm.2.4}, for $\alpha\in (1,2)$, \eqref{eq.firstone} 
implies $x(t)/F^{-1}(Mt)\to 1$ as $t\to\infty$. By Karamata's theorem, the function in the limit in \eqref{eq.firstone}
is in $\text{RV}_{\infty}((1-\theta)/(1+\theta)+2-\alpha)$, and the index is negative for the range of $\alpha\in (1,2)$ 
stated in \eqref{eq.lim1}. When $\alpha=2$, \eqref{eq.suff11} is still equivalent to \eqref{eq.firstone}, and the index of 
regular variation is negative because $\theta>1$. Hence we have shown \eqref{eq.lim1}.

We now prove \eqref{eq.lim2}. By \eqref{eq.intsareRV} and \eqref{eq.fFasymptoticsex}, as $x\to\infty$ 
\[
\int_{[F(x)/M,\infty)} \mu(ds) \sim \int_{\frac{1}{M(\theta+1)}\log^{1+\theta} x}^\infty k(s)\,ds
\sim  \left(\frac{1}{M(\theta+1)}\right)^{1-\alpha}\int_{\log^{1+\theta} x}^\infty k(s)\,ds
\]
and 
\[
 \frac{F(x)}{M}\int_{[F(x)/M,\infty)} \mu(du) \sim 
(\log x)^{1+\theta} \left(\frac{1}{M(\theta+1)}\right)^{2-\alpha}\int_{\log^{1+\theta} x}^\infty k(s)\,ds.
\]
Therefore by \eqref{eq.mumomemttail}, \eqref{eq.suff10} is equivalent to 
\begin{align*}
\min\left(\log x \cdot \int_{\log^{1+\theta} x}^\infty k(s)\,ds,
\frac{1}{\log^\theta x}  \int_{0}^{\log^{1+\theta} x} sk(s)\,ds\right) \to +\infty, \quad x\to\infty.
\end{align*}
Hence \eqref{eq.suff10} is equivalent to 
\begin{equation}   \label{eq.zeroconditionex2}
\min\left(
t^{1/(1+\theta)}\cdot \int_{t}^\infty k(s)\,ds,  t^{-\frac{\theta}{1+\theta}} \int_{0}^{t} sk(s)\,ds\right) \to +\infty
\text{ as $t\to\infty$},
\end{equation} 
and this implies $x(t)/F^{-1}(Mt)\to 0$ as $t\to\infty$.
Both functions in the minimum are in $\text{RV}_\infty(1/(1+\theta) - \alpha+1)$. Therefore, if $\alpha$ is in the interval specified in \eqref{eq.lim2}, we have that the index of regular variation is positive, and therefore \eqref{eq.zeroconditionex2} holds. This proves the required asymptotic behaviour in \eqref{eq.lim2}.
\end{examples}
\begin{examples}
We now present a simple application of Theorem \ref{thm.2.2.1.1} again with $f$ as in Example  \ref{example.1}. Since Theorem \ref{thm.2.2.1.1} deals with the case when $\lambda = \infty$ we must have $\theta \in (0,1)$. We have shown already that $f$ obeys both $0 < f'(x) \to 0$ as $x\to\infty$ and $f$  decreasing on $[x_2,\infty)$ for some $x_2>0$. Hence the unique continuous solution, $x$, of \eqref{eq.fde} obeys
\begin{align*}
x(t) \sim F^{-1}\left(Mt - c(t)\log F^{-1}(Mt)\right) \sim 
\exp\left( (\theta+1)^{\frac{1}{1+\theta}}\left[ Mt - c(t) (Mt)^{\frac{1}{1+\theta}} \right]^{\frac{1}{1+\theta}} \right), \mbox{ as }t\to\infty,
\end{align*}
where $\lim_{t\to\infty} c(t) = C(1+\theta)^{1/(1+\theta)}$. It is instructive to rewrite the above expression in the form
\begin{align}\label{example.equiv.rewrite}
x(t) &\sim
\exp\left( (\theta+1)^{\frac{1}{1+\theta}} \left[ (Mt)^{1/(1+\theta)} - \tilde{c}(t) (Mt)^{(1-\theta)/(1+\theta)} \right] \right)\nonumber\\
&= y(t)\exp\left( -(\theta+1)^{\frac{1}{1+\theta}}\tilde{c}(t) (Mt)^{(1-\theta)/(1+\theta)} \right), \mbox{ as }t\to\infty,
\end{align}
where a simple application of the mean value theorem shows that $\tilde{c}(t) \sim C\left\{(\theta+1)\right\}^{-1/(1+\theta)}$ and $y(t)$ is the solution to \eqref{eq.equivode} with unit initial condition. Restating the conclusion of Theorem \ref{thm.2.2.1.1} in the form \eqref{example.equiv.rewrite} shows explicitly that the solution of \eqref{eq.fde} is asymptotic to the solution of \eqref{eq.equivode} times a retarding factor which tends to zero as $t\to\infty$. Notice that the main term in the exponent in the retarding factor is of the order $t^{(1-\theta)/(1+\theta)}$; from Example \ref{example.1}, 
the corresponding growth term in $y$ is of the order $t^{1/(1+\theta)}$. Since $\theta\in (0,1)$ the solution $x$ still grows, at a rate roughly described by $\exp(K t^{\theta/(1+\theta)})$. 
\end{examples}

\section{An Implicit Asymptotic Relation} \label{sec.asylemma}
We state and prove two key lemmata which enable direct asymptotic information to be obtained for solutions of 
\eqref{eq.fde} and \eqref{eq.vde} from the indirect asymptotic relation
\begin{equation} \label{eq.Fxlogf}
\lim_{t\to\infty} \frac{F(x(t))-Mt}{\log f(x(t))}= -C.    
\end{equation}
In the first result, $C$ is finite: in the second, $C=+\infty$. 
\begin{lemma} \label{lemma.l1}
Let $M>0$, $C\in (0,\infty)$. Suppose $x(t)\to\infty$ as $t\to\infty$ is such that \eqref{eq.Fxlogf} holds with $C\in [0,\infty)$ and $f$ is increasing on $[x_1,\infty)$ and obeys \eqref{eq.lambda} with $\lambda\in [0,\infty]$. If $x$ also obeys  
\begin{equation}  \label{eq.elemlimsupl}
\limsup_{t\to\infty} \frac{x(t)}{F^{-1}(Mt)}\leq 1,
\end{equation}
then 
\[
\lim_{t\to\infty} \frac{x(t)}{F^{-1}(Mt)}=e^{-\lambda C}.
\]
\end{lemma}
\begin{proof}
We consider separately the cases where $\lambda\in(0,\infty)$, $\lambda=0$ and $\lambda=+\infty$.

\textbf{Case I: $\lambda=0$.}
In the case $\lambda=0$, we have 
\[
\limsup_{x\to\infty} \frac{\log f(x)}{\log x}\leq 1.
\]
Therefore by \eqref{eq.Fxlogf}
\[
\limsup_{t\to\infty} \frac{Mt-F(x(t))}{\log x(t)}
=\limsup_{t\to\infty} \frac{Mt-F(x(t))}{\log f(x(t))}\cdot \frac{\log f(x(t))}{\log x(t)}\leq C.
\]
Hence
\begin{equation} \label{eq.L0}
L_0:=
\liminf_{t\to\infty} \frac{F(x(t))-Mt}{\log x(t)}\geq -C.
\end{equation}
Thus, for every $\epsilon>0$, there is $T_3>0$ such that for $t\geq T_3$ we have 
$(F(x(t))-Mt)/\log x(t) >-C-1 = -(C+1)$. Hence with $3\mu^\ast/4:=C+1>0$ we have
\begin{equation} \label{eq.x2}
F(x(t)) +\frac{3}{4}\mu^\ast \log x(t) > Mt, \quad t\geq T_3.
\end{equation}
Recall the  estimate \eqref{eq.elemlimsupl}. Suppose, in contradiction to the conclusion when $\lambda=0$, 
that 
\begin{equation} \label{eq.x1}
\liminf_{t\to\infty} \frac{x(t)}{F^{-1}(Mt)}=\underline{\Lambda}\in [0,1).
\end{equation}
Since $\underline{\Lambda}\in [0,1)$, there is $\epsilon_0>0$ such that 
\[
\underline{\Lambda} + \epsilon < e^{-\epsilon \mu^\ast}, \quad \epsilon<\epsilon_0.
\]
Define $\varphi(\epsilon)=e^{-\epsilon \mu^\ast}$. 
By \eqref{eq.x1}, if $\underline{\Lambda}\in [0,1)$, for all $\epsilon\in (0,\epsilon_0)$ there is a sequence $\tau_n^\epsilon\uparrow \infty$ as $n\to\infty$ such that 
\[
x(\tau_n^\epsilon)<(\underline{\Lambda}+\epsilon) F^{-1}(M\tau_n^\epsilon) < \varphi(\epsilon) F^{-1}(M\tau_n^\epsilon)
=:v_n^\epsilon.
\]
Since $\tau_n^\epsilon\uparrow \infty$, it follows that there is $N_1\in \mathbb{N}$ such that 
$\tau_n^\epsilon>T_4$ for all $n>N_1$. Hence for $n>N_1$ we have from \eqref{eq.x2}
\[
F(x(\tau_n^\epsilon)) +\frac{3}{4}\mu^\ast \log x(\tau_n^\epsilon) > M\tau_n^\epsilon. 
\]
Now $x(\tau_n^\epsilon)<v_n^\epsilon$. Hence for $n>N_1$
\[
M\tau_n^\epsilon < F(x(\tau_n^\epsilon)) +\frac{3}{4}\mu^\ast \log x(\tau_n^\epsilon)  < F(v_n^\epsilon) + \frac{3}{4}\mu^\ast \log v_n^\epsilon.
\]
Since $M\tau_n^\epsilon = F(v_n^\epsilon/\varphi(\epsilon))$, so
\begin{equation} \label{eq.v}
F(v_n^\epsilon/\varphi(\epsilon)) < F(v_n^\epsilon) + \frac{3}{4}\mu^\ast \log v_n^\epsilon, \quad n>N_1.
\end{equation}
We wish to show that \eqref{eq.v} is impossible. If we can show that 
\begin{equation}  \label{eq.f6}
\text{There is $x_3(\epsilon)>0$ such that }
F(x/\varphi(\epsilon))-F(x)-\frac{3}{4}\mu^\ast \log x>0, \quad x>x_3(\epsilon),
\end{equation}
we may take $v_n^\epsilon>x_3(\epsilon)$ (which will be true for all $n>N_2(\epsilon)$), so that for $n>N_3=\max(N_1,N_2)$ we have 
\[
F(v_n^\epsilon/\varphi(\epsilon))-F(v_n^\epsilon)-\frac{3}{4}\mu^\ast \log v_n^\epsilon
>0>
F(v_n^\epsilon/\varphi(\epsilon))-F(v_n^\epsilon)-\frac{3}{4}\mu^\ast \log v_n^\epsilon,
\]
where we used \eqref{eq.f6} to get the first inequality, and \eqref{eq.v} to get the second. This generates the required contradiction. Therefore, it suffices to prove \eqref{eq.f6}.

Since $f(x)=o(x/\log x)$, for every $\epsilon\in (0,\epsilon_0)$ there is an $x_3(\epsilon)>0$ such that 
$f(x)<\epsilon x/\log x$ for $x\geq x_3(\epsilon)$. Thus for $x\geq x_3(\epsilon)$ we get 
\[
\int_x^{x/\varphi(\epsilon)} \frac{1}{f(u)}\,du 
\geq \frac{1}{\epsilon}\int_x^{x/\varphi(\epsilon)} \frac{\log u}{u}\,du 
\geq \frac{\log x}{\epsilon}\int_x^{x/\varphi(\epsilon)} \frac{1}{u}\,du. 
\]
Hence for $x\geq x_3(\epsilon)$, from the fact $\varphi(\epsilon)=e^{-\mu^\ast \epsilon}$, we get that
\[
\frac{1}{\log x}
\int_x^{x/\varphi(\epsilon)} \frac{1}{f(u)}\,du 
\geq \frac{1}{\epsilon}\left(\log(x/\varphi(\epsilon))-\log(x)\right)
=\frac{1}{\epsilon}\log\left(\frac{1}{\varphi(\epsilon)}\right)=\mu^\ast.
\]
Since
\[
F(x/\varphi(\epsilon))-F(x)-\frac{3}{4}\mu^\ast \log x
=\log x\left( \frac{1}{\log x}\int_x^{x/\varphi(\epsilon)} \frac{1}{f(u)}\,du - \frac{3}{4}\mu^\ast \right),
\]
for $x\geq x_3(\epsilon)$ we have 
\[
F(x/\varphi(\epsilon))-F(x)-\frac{3}{4}\mu^\ast \log x
\geq \log x \frac{\mu^\ast}{4}>0.
\]
This is \eqref{eq.f6}. Hence, in contradiction to \eqref{eq.x1} we have
\[
\liminf_{t\to\infty} \frac{x(t)}{F^{-1}(Mt)}\geq 1.
\]
Combining this with \eqref{eq.elemlimsupl} we get 
\[
\lim_{t\to\infty} \frac{x(t)}{F^{-1}(Mt)}=1=e^{-\lambda C},
\]
because $\lambda=0$. We have therefore proven the result in the case $\lambda=0$.

\textbf{Case II: $\lambda\in (0,\infty)$.}
In this case, we have that 
\[
\lim_{x\to\infty} \frac{\log f(x)}{\log x}=1.
\]
Therefore, from \eqref{eq.Fxlogf}, we get
\[
\lim_{t\to\infty} \frac{F(x(t))-Mt}{\log x(t)}=-C,
\]
and so, for every $\epsilon\in (0,1)$ there is a $T_3(\epsilon)>0$ such that
\begin{equation} \label{eq.17}
-C(1+\epsilon)\log x(t)<F(x(t))-Mt < -C(1-\epsilon)\log x(t), \quad t\geq T_3(\epsilon).
\end{equation}
By \eqref{eq.elemlimsupl}, we have $\bar{\Lambda}:=\limsup_{t\to\infty} x(t)/F^{-1}(Mt)\leq 1$.
Suppose that 
\begin{equation} \label{eq.18}
e^{-\lambda C} < \bar{\Lambda} \leq 1.
\end{equation}
Since $\bar{\Lambda}>e^{-\lambda C}$ there is $\epsilon_0<1/2$ such that 
\begin{equation} \label{eq.19}
e^{3C\epsilon \lambda}<\frac{\bar{\Lambda}}{e^{-\lambda C}}, \quad \epsilon<\epsilon_0.
\end{equation}
By \eqref{eq.18}, for every $\epsilon\in (0,\epsilon_0\wedge 1/2)$, there is a sequence $t_n^\epsilon\uparrow\infty$
such that 
\[
x(t_n^\epsilon) > \bar{\Lambda} e^{-\epsilon C\lambda} F^{-1}(Mt_n^\epsilon),
\]
so by \eqref{eq.19}, $x(t_n^\epsilon) >  e^{-C\lambda} e^{2\epsilon C\lambda} F^{-1}(Mt_n^\epsilon)$. Put $\varphi(\epsilon)=e^{2C\lambda\epsilon}$. Since $t_n^\epsilon\uparrow\infty$, it follows that there is $N_1(\epsilon)\in \mathbb{N}$ such that $t_{N_1}^\epsilon>T_3(\epsilon)$. Thus  $t_n^\epsilon>T_3(\epsilon)$ for all $n\geq N_1(\epsilon)$. Define $u_n^\epsilon=e^{-\lambda C}\varphi(\epsilon)F^{-1}(Mt_n^\epsilon)$. Then $x(t_n^\epsilon)>u_n^\epsilon$ and 
$F(e^{\lambda C}u_n^\epsilon/\varphi(\epsilon))=Mt_n^\epsilon$. We see also that $u_n^\epsilon\to\infty$ as $n\to\infty$. 

Next, as $f(x)\sim \lambda x/\log x$ as $x\to\infty$, we can show that  
\[
\lim_{x\to\infty} \frac{1}{x/f(x)}\int_{xe^{\lambda C}/\varphi(\epsilon)}^x \frac{1}{f(u)}\,du
=-\log\left(\frac{e^{\lambda C}}{\varphi(\epsilon)} \right)=-\lambda C+2\epsilon \lambda C.
\] 
Therefore
\[
\lim_{x\to\infty} \left\{\frac{1}{\log x}\int_{xe^{\lambda C}/\varphi(\epsilon)}^x \frac{1}{f(u)}\,du + C(1-\epsilon)\right\}
=C\epsilon.
\]
Thus for every $\eta\in (0,1/2)$ there is $\tilde{x}_3(\eta,\epsilon)>0$ such  that $x>\tilde{x}_3(\eta,\epsilon)$ 
implies 
\[
C(1-\epsilon) + \frac{1}{\log x}\int_{xe^{\lambda C}/\varphi(\epsilon)}^x \frac{1}{f(u)}\,du > C\epsilon(1-\eta).
\]
Put $\eta=1/4$ and let $x_3(\epsilon)=\tilde{x}_3(1/4,\epsilon)$. Then for $x>x_3(\epsilon)$ we have 
\[
C(1-\epsilon) + \frac{1}{\log x}\int_{xe^{\lambda C}/\varphi(\epsilon)}^x \frac{1}{f(u)}\,du > C\epsilon \frac{3}{4}>0.
\] 
Next, as $u_n^\epsilon\to\infty$ as $n\to\infty$, there is $N_2(\epsilon)\in\mathbb{N}$ such that $u_n^\epsilon>x_3(\epsilon)>1$ for all $n\geq N_2(\epsilon)$. Let $N_3(\epsilon)=\max(N_1,N_2)$. Then for $n\geq N_3(\epsilon)$ we have 
\begin{equation} \label{eq.20}
C(1-\epsilon) + \frac{1}{\log u_n^\epsilon}\int_{u_n^\epsilon e^{\lambda C}/\varphi(\epsilon)}^{u_n^\epsilon} \frac{1}{f(u)}\,du >0.
\end{equation}
Since $t_n^\epsilon>T_3(\epsilon)$ for all $n\geq N_3(\epsilon)$, $x(t_n^\epsilon)>e^{-\lambda C}\varphi(\epsilon)F^{-1}(Mt_n^\epsilon)$, and so $x(t_n^\epsilon)>u_n^\epsilon$. By \eqref{eq.17}, as $t_n^\epsilon>T_3(\epsilon)$ and $F$ and $x\mapsto \log (x)$ are increasing, we have
\begin{align*}
0&>F(x(t_n^\epsilon))-Mt_n^\epsilon+C(1-\epsilon)\log x(t_n^\epsilon)\\
&>F(u_n^\epsilon)-M t_n^\epsilon+C(1-\epsilon)\log u_n^\epsilon \\
&=F(u_n^\epsilon)-F(e^{\lambda c}u_n^\epsilon/\varphi(\epsilon)) +C(1-\epsilon)\log u_n^\epsilon \\
&=\int_{u_n^\epsilon e^{\lambda C}/\varphi(\epsilon)}^{u_n^\epsilon} \frac{1}{f(u)}\,du
 +C(1-\epsilon)\log u_n^\epsilon \\
&= \log u_n^\epsilon\left\{ C(1-\epsilon) 
+ \frac{1}{\log u_n^\epsilon}\int_{u_n^\epsilon e^{\lambda C}/\varphi(\epsilon)}^{u_n^\epsilon} \frac{1}{f(u)}\,du\right\} >0,
\end{align*}
where we used \eqref{eq.20} at the last step. This gives the desired contradiction to \eqref{eq.18}. Hence we must have 
\begin{equation} \label{eq.21} 
\limsup_{t\to\infty} \frac{x(t)}{F^{-1}(Mt)}\leq e^{-\lambda C}.
\end{equation}
Next we suppose that 
\begin{equation} \label{eq.22} 
\liminf_{t\to\infty} \frac{x(t)}{F^{-1}(Mt)}=: \underline{\Lambda}< e^{-\lambda C}.
\end{equation}
Recall from \eqref{eq.17} that 
\[
F(x(t))-Mt + C(1+\epsilon)\log x(t)>0, \quad t>T_3(\epsilon).
\]
Let $\varphi_2(\epsilon)=e^{-2\epsilon C \lambda}$. Since $\underline{\Lambda}<e^{-\lambda C}$ and 
$\varphi_2(\epsilon)\to 1$ as $\epsilon\to 0^+$, there is $\epsilon_1<1/2$ such that $\epsilon<\epsilon_1$ implies 
$\underline{\Lambda}+\epsilon<e^{-\lambda C}\varphi_2(\epsilon)$. By \eqref{eq.22}, it follows that there is $\tau_n^\epsilon\uparrow\infty$ such that 
\[
x(\tau_n^\epsilon)<(\underline{\Lambda}+\epsilon)F^{-1}(M\tau_n^\epsilon) <e^{-\lambda C}\varphi_2(\epsilon) F^{-1}(M\tau_n^\epsilon)
\]
Since $\tau_n^\epsilon\to\infty$ as $n\to\infty$, there is an $N_4(\epsilon)\in\mathbb{N}$ such that $\tau_n^\epsilon>T_3(\epsilon)$ for all $n\geq N_4(\epsilon)$. Define $v_n^\epsilon=e^{-\lambda C}\varphi_2(\epsilon) F^{-1}(M\tau_n^\epsilon)$, so $x(\tau_n^\epsilon)>v_n^\epsilon$ and $F(e^{\lambda C} v_n^\epsilon/\varphi_2(\epsilon))=M\tau_n^\epsilon$. Next, 
$v_n^\epsilon\to\infty$ as $n\to\infty$ and we get as before
\[
\lim_{x\to\infty} \frac{1}{x/f(x)}\int_{xe^{\lambda C}/\varphi_2(\epsilon)}^x \frac{1}{f(u)}\,du
=-\lambda C+\log\varphi_2(\epsilon).
\]
Thus, as $f(x)/(x/\log x) \to \lambda$ as $x\to\infty$, and $\log\varphi_2(\epsilon)=-2C\lambda \epsilon$, we get
\[
\lim_{x\to\infty} \left\{\frac{1}{\log x}\int_{xe^{\lambda C}/\varphi_2(\epsilon)}^x \frac{1}{f(u)}\,du
+C(1+\epsilon)\right\}
=-C\epsilon. 
\]
Therefore, for every $\eta\in (0,1/2)$ there exists $\tilde{x}_4(\eta, \epsilon)>0$ such that $x>\tilde{x}_4(\eta,\epsilon)$ implies 
\[
C(1+\epsilon)+ \frac{1}{\log x}\int_{xe^{\lambda C}/\varphi_2(\epsilon)}^x \frac{1}{f(u)}\,du <-C\epsilon + C\epsilon\eta.
\] 
Put $\eta=1/4$, and let $x_4(\epsilon)=\tilde{x}_4(1/4,\epsilon)$. Then for $x>x_4(\epsilon)$
\[
C(1+\epsilon)+ \frac{1}{\log x}\int_{xe^{\lambda C}/\varphi_2(\epsilon)}^x \frac{1}{f(u)}\,du <-\frac{3}{4}C\epsilon<0.
\]
Since $v_n^\epsilon\to\infty$ as $n\to\infty$, there is $N_5(\epsilon)\in\mathbb{N}$ such that $v_n^\epsilon>x_4(\epsilon)>1$ for all $n\geq N_5(\epsilon)$. Let $N_6(\epsilon)=\max(N_4(\epsilon), N_5(\epsilon))$. Then for $n\geq N_6(\epsilon)$ we have
\begin{equation} \label{eq.23}
C(1+\epsilon)+\frac{1}{\log v_n^\epsilon}\int_{v_n^\epsilon e^{\lambda C}/\varphi_2(\epsilon)}^{v_n^\epsilon} \frac{1}{f(u)}\,du<0.
\end{equation}
Since $\tau_n^\epsilon>T_3(\epsilon)$ for all $n\geq N_6(\epsilon)$, $x(\tau_n^\epsilon)>v_n^\epsilon$, and $F$ and 
$x\mapsto \log x$ are increasing, by \eqref{eq.17} we have 
\begin{align*}
0&<F(x(\tau_n^\epsilon))-M\tau_n^\epsilon + C(1+\epsilon)\log x(\tau_n^\epsilon)\\
&<F(v_n^\epsilon)-M\tau_n^\epsilon + C(1+\epsilon)\log v_n^\epsilon\\
&=F(v_n^\epsilon)-F(e^{\lambda C}v_n^\epsilon/\varphi_2(\epsilon))+C(1+\epsilon)\log v_n^\epsilon\\
&=\int_{v_n^\epsilon e^{\lambda C}/\varphi_2(\epsilon)}^{v_n^\epsilon} \frac{1}{f(u)} +C(1+\epsilon)\log v_n^\epsilon\\
&=\log v_n^\epsilon \left\{\frac{1}{\log v_n^\epsilon}
\int_{v_n^\epsilon e^{\lambda C}/\varphi_2(\epsilon)}^{v_n^\epsilon} \frac{1}{f(u)} +C(1+\epsilon)\right\}\\
&<0,
\end{align*}
by \eqref{eq.23}, a contradiction. Hence the supposition \eqref{eq.22} is false. Thus
\[
\liminf_{t\to\infty} \frac{x(t)}{F^{-1}(Mt)}\geq e^{-\lambda C}.
\] 
Combining this and \eqref{eq.21} gives 
\begin{equation} \label{eq.24}
\lim_{t\to\infty} \frac{x(t)}{F^{-1}(Mt)}= e^{-\lambda C},
\end{equation}
as desired. This completes the proof when $\lambda\in (0,\infty)$.

\textbf{Case III: $\lambda=+\infty$.} In this case, we have that $f(x)/x\to 0$ as $x\to\infty$ and 
$f(x)/(x/\log x)\to\infty$ as $x\to\infty$, so therefore 
$\log f(x)/\log x\to 1$ as $x\to\infty$. Hence, from \eqref{eq.Fxlogf}, we get
\[
\lim_{t\to\infty} \frac{F(x(t))-Mt}{\log x(t)}=-C,
\]
and so, for every $\epsilon\in (0,1/2)$ there is a $T_3(\epsilon)>0$ such that \eqref{eq.17} holds, i.e.,
\begin{equation*} 
-C(1+\epsilon)\log x(t)<F(x(t))-Mt < -C(1-\epsilon)\log x(t), \quad t\geq T_3(\epsilon).
\end{equation*}
Recall the estimate \eqref{eq.elemlimsupl}. Suppose, in contradiction to the conclusion when $\lambda=+\infty$, 
that 
\begin{equation} \label{eq.25}
\liminf_{t\to\infty} \frac{x(t)}{F^{-1}(Mt)}=\bar{\Lambda}\in (0,1].
\end{equation}
There is a sequence $t_n^\epsilon\uparrow \infty$ as $n\to\infty$ such that 
\[
x(t_n^\epsilon)>\bar{\Lambda}(1-\epsilon) F^{-1}(Mt_n^\epsilon) > K(\epsilon) F^{-1}(Mt_n^\epsilon)=:u_n^\epsilon,
\]
where $K(\epsilon)\in (0,\bar{\Lambda}(1-\epsilon))\subset(0,1)$. 
Since $t_n^\epsilon\uparrow \infty$, it follows that there is $N_1(\epsilon)\in \mathbb{N}$ such that 
$t_n^\epsilon>T_3(\epsilon)$ for all $n\geq N_1(\epsilon)$. Hence for $n\geq N_1(\epsilon)$ we have
\begin{equation}  \label{eq.281}
F(x(t_n^\epsilon))-Mt_n^\epsilon < -C(1-\epsilon)\log x(t_n^\epsilon).
\end{equation}

Since $K(\epsilon)<1$ and $f$ is increasing, we have 
\[
\frac{1}{\log x}\int_{x}^{x/K(\epsilon)} \frac{1}{f(u)}\,du < (K(\epsilon)^{-1}-1)\frac{x}{f(x)\log x}.
\]
Since $f(x)/(x/\log x)\to\infty$ as $x\to\infty$, letting $x\to\infty$ gives
\begin{equation*} 
\lim_{x\to\infty} \frac{1}{\log x} \int_{x}^{x/K(\epsilon)} \frac{1}{f(u)}\,du =0.
\end{equation*}
Therefore, for every $\eta\in (0,1/2)$, there is $\tilde{x}_5(\eta,\epsilon)$ such that $x>\tilde{x}_5(\eta,\epsilon)$ 
implies 
\[
\frac{1}{\log x} \int_{x}^{x/K(\epsilon)} \frac{1}{f(u)}\,du < C\eta.
\]
Pick $\eta=\epsilon$, and set $x_5(\epsilon)=\tilde{x}_5(\epsilon,\epsilon)$. Then for $x\geq x_5(\epsilon)$ we have 
\[
\frac{1}{\log x} \int_{x}^{x/K(\epsilon)} \frac{1}{f(u)}\,du < C\epsilon.
\]
Since $u_n\to\infty$ as $n\to\infty$, there is $N_2(\epsilon)\in\mathbb{N}$ such that for $n\geq N_2(\epsilon)$ 
we have $u_n^\epsilon>x_5(\epsilon)$. Hence
\begin{equation} \label{eq.282}
\frac{1}{\log u_n^\epsilon} \int_{u_n^\epsilon}^{u_n^\epsilon/K(\epsilon)} \frac{1}{f(u)}\,du < C\epsilon, \quad n\geq N_2(\epsilon).
\end{equation}
Finally, let $N_3(\epsilon)=\max(N_1(\epsilon),N_2(\epsilon))$. 
Since $u_n^\epsilon<x(t_n^\epsilon)$, we have $F(u_n^\epsilon)<F(x(t_n^\epsilon))$ and $\log u_n^\epsilon<\log x(t_n^\epsilon)$. Therefore by \eqref{eq.281} and \eqref{eq.282}
\begin{align*}
0&>F(x(t_n^\epsilon)) +C(1-\epsilon)\log x(t_n^\epsilon) -Mt_n^\epsilon\\
&> F(u_n^\epsilon) + C(1- \epsilon)\log u_n^\epsilon-Mt_n^\epsilon\\
&= F(u_n^\epsilon) + C(1- \epsilon)\log u_n^\epsilon-F(u_n^\epsilon /K(\epsilon))\\
&= C(1- \epsilon)\log u_n^\epsilon 
- \int_{u_n^\epsilon}^{u_n^\epsilon/K(\epsilon)} \frac{1}{f(u)}\,du\\
&=\log u_n^\epsilon \left\{ C(1-\epsilon) - \frac{1}{\log u_n^\epsilon} \int_{u_n^\epsilon}^{u_n^\epsilon/K(\epsilon)} \frac{1}{f(u)}\,du\right\}\\
&>\log u_n^\epsilon ( C(1-\epsilon) - C\epsilon)\\
&=\log u_n^\epsilon C(1-2\epsilon)>0,
\end{align*}
 a contradiction. Hence the supposition \eqref{eq.25} is false, and we have
$x(t)/F^{-1}(Mt)\to 0$ as $t\to\infty$ as claimed. 
\end{proof}
For the Volterra equation \eqref{eq.vde}, we will need a new variant of Lemma~\ref{lemma.l1} to cover the case when 
\[
\int_{[0,\infty)} s\mu(ds)=+\infty.
\] 
\begin{lemma} \label{lemma.l8}
Let $M>0$. Suppose $x(t)\to\infty$ as $t\to\infty$ is such that 
\begin{equation} \label{eq.minusinftyCimplicit}
\lim_{t\to\infty} \frac{F(x(t))-Mt}{\log x(t)} =-\infty.
\end{equation}
Suppose also $f$ is increasing and obeys \eqref{eq.lambda} with $\lambda\in (0,\infty]$ and $f'(x)\to 0$ 
as $x\to\infty$. If $x$ also obeys  
  \eqref{eq.elemlimsupl} then 
\[
\lim_{t\to\infty} \frac{x(t)}{F^{-1}(Mt)}=0.
\]
\end{lemma}
\begin{proof}
From \eqref{eq.minusinftyCimplicit}, we are free to prepare the estimate 
\begin{align} \label{eq.3.2.6pr}
\text{For every $\epsilon\in (0,1)$ there is $T_3(\epsilon)>0$ such that } 
F(x(t))+\frac{2}{\epsilon}\log x(t) - Mt<0, \quad t\geq T_3(\epsilon)
\end{align}
for later use. We now proceed to derive the result that $x(t)/F^{-1}(Mt)\to 0$ as $t\to\infty$ by 
emulating the proof of Lemma~\ref{lemma.l1}.
Suppose not. Then, in view of \eqref{eq.elemlimsupl}, we have 
\begin{equation} \label{eq.3.2.6}
\limsup_{t\to\infty} \frac{x(t)}{F^{-1}(Mt)}=:\bar{\Lambda}\in (0,1].
\end{equation}
Then there is a sequence $t_n\uparrow \infty$ as $n\to\infty$ such that 
$x(t_n^\epsilon) >\bar{\Lambda}(1-\epsilon) F^{-1}(Mt_n^\epsilon)>K(\epsilon)F^{-1}(Mt_n^\epsilon)$ where 
$K(\epsilon)\in (0,\bar{\Lambda}(1-\epsilon))\subset (0,1)$. Since $t_n^\epsilon\uparrow \infty$ as $n\to\infty$, there is $N_1(\epsilon)\in \mathbb{N}$ such that $t_n^\epsilon>T_3(\epsilon)$ for all $n\geq N_1(\epsilon)$. Define $u_n^\epsilon=K(\epsilon)F^{-1}(Mt_n^\epsilon)$. Then $x(t_n^\epsilon)>u_n^\epsilon$ and 
$F(u_n^\epsilon/K(\epsilon))=Mt_n^\epsilon$. Moreover $u_n^\epsilon\to\infty$ as $n\to\infty$. If $\lambda=+\infty$, take $K(\epsilon)=\bar{\Lambda}(1-\epsilon)/2$, while if $\lambda\in (0,\infty)$, take 
$K(\epsilon)=e^{-\lambda(1/\epsilon-1)}$. There is $\epsilon_0\in (0,1)$ such that 
$e^{-\lambda(1/\epsilon-1)}<\bar{\Lambda}(1-\epsilon)$ for all $\epsilon<\epsilon_0\wedge 1$.     

In the case that $\lambda\in (0,\infty)$, it is a direct calculation to show that  
\begin{equation} \label{eq.3.2.7pr}
\lim_{x\to\infty} \frac{1}{\log x} \int_{x}^{x/K(\epsilon)} \frac{1}{f(u)}\,du 
= \frac{1}{\lambda}\log\left(\frac{1}{K(\epsilon)}\right).
\end{equation}
If $\lambda=+\infty$, since $f$ is increasing on $[x_1,\infty)$, for $x>x_1$ we have
\[
0<\frac{1}{\log x}\int_{x}^{x/K(\epsilon)} \frac{1}{f(u)}\,du \leq \left(\frac{1}{K(\epsilon)}-1\right)
\frac{x/\log x}{f(x)},
\]
so as $(x/\log x)/f(x)\to 0$ as $x\to\infty$, we get 
\begin{equation*} 
\lim_{x\to\infty} \frac{1}{\log x} \int_{x}^{x/K(\epsilon)} \frac{1}{f(u)}\,du = 0. 
\end{equation*}
Hence combining this estimate with \eqref{eq.3.2.7pr} we get
\begin{equation} \label{eq.3.2.7}
\lim_{x\to\infty} \frac{1}{\log x} \int_{x}^{x/K(\epsilon)} \frac{1}{f(u)}\,du 
=\left\{
\begin{array}{cc}
-\frac{1}{\lambda} \log K(\epsilon), & \lambda \in (0,\infty), \\ 
0, & \lambda = +\infty
\end{array}
\right.
\end{equation}
We seek to obtain a consolidated estimate covering these cases. Let $\epsilon\in (0,\epsilon_0\wedge 1)$. When $\lambda=+\infty$, it is clear there is $x_3(\epsilon)>1$ such that 
\[
\int_{x}^{x/K(\epsilon)} \frac{1}{f(u)}\,du < \log x < \frac{2}{\epsilon}\log x, \quad x\geq x_3(\epsilon).
\]
For $\lambda\in (0,\infty)$, there is $x_3(\epsilon)>1$ such that for $x\geq x_3(\epsilon)$ we have 
\[
\int_{x}^{x/K(\epsilon)} \frac{1}{f(u)}\,du < \left(1-\frac{1}{\lambda} \log K(\epsilon)\right) \log x
=\frac{1}{\epsilon} \log x, 
\]   
where we used the definition of $K(\epsilon)$ to obtain the last equality. Therefore we see for every $\epsilon<\epsilon_0\wedge 1$ that there is $x_3(\epsilon)>1$ such that 
\begin{equation} \label{eq.3.2.8}
\int_{x}^{x/K(\epsilon)} \frac{1}{f(u)}\,du  < \frac{2}{\epsilon}\log x, \quad x\geq x_3(\epsilon),
\end{equation} 
regardless as to whether $\lambda\in (0,\infty]$. Therefore this implies for $x\geq x_3(\epsilon)$ that 
\[
F(x/K(\epsilon))-F(x)-\frac{2}{\epsilon}\log x = \int_{x}^{x/K(\epsilon)} \frac{1}{f(u)}\,du - 
\frac{2}{\epsilon}\log x<0.
\]
Therefore as $u_n^\epsilon\to\infty$ as $n\to\infty$, there is $N_2(\epsilon)\in \mathbb{N}$ such that 
for $n\geq N_2(\epsilon)$ we have $u_n^\epsilon>x_3(\epsilon)$. Thus with 
$n\geq N_3(\epsilon):=\max(N_1(\epsilon),N_2(\epsilon))$ we have 
\begin{equation} \label{eq.3.2.9}
F(u_n^\epsilon/K(\epsilon))-F(u_n^\epsilon)-\frac{2}{\epsilon}\log u_n^\epsilon < 0.
\end{equation} 
On the other hand, as $n\geq N_3(\epsilon)\geq N_1(\epsilon)$ and $t_n^\epsilon>T_3(\epsilon)$ for 
$n\geq N_1(\epsilon)$, we have from \eqref{eq.3.2.6pr} that
\begin{equation} \label{eq.3.2.10}
F(x(t_n^\epsilon)) - Mt_n^\epsilon + \frac{2}{\epsilon} \log x(t_n^\epsilon)<0.
\end{equation}  
Therefore for $n\geq N_3(\epsilon)$, since $F(u_n^\epsilon/K(\epsilon))=Mt_n^\epsilon$ and $x(t_n^\epsilon)>u_n^\epsilon$ we get from \eqref{eq.3.2.9} and \eqref{eq.3.2.10} that
\begin{align*} 
0&>F(x(t_n^\epsilon)) - Mt_n^\epsilon + \frac{2}{\epsilon} \log x(t_n^\epsilon) \\
&= -F(u_n^\epsilon/K(\epsilon)) +F(x(t_n^\epsilon))  + \frac{2}{\epsilon} \log x(t_n^\epsilon)\\
&> -F(u_n^\epsilon/K(\epsilon)) +F(u_n^\epsilon)  + \frac{2}{\epsilon} \log u_n^\epsilon>0,
\end{align*}
which is a contradiction, and the monotonicity of $x\mapsto F(x)+\epsilon^{-1}\log x$ was used at the penultimate step. This implies that \eqref{eq.3.2.6} is false, so we must have $\limsup_{t\to\infty} x(t)/F^{-1}(Mt)=0$,
as claimed.   
\end{proof}

\section{Proof of Theorem~\ref{thm.2.2}}
Our hypotheses on $\psi$ and the positivity of $f$ immediately yield that $x(t)\to\infty$ as $t\to\infty$. Thus there exists $T_1$ such that $x(t)> x_1$ for all $t \geq T_1$. Letting $t > T_1 + \tau$, and noting that $t \mapsto x(t)$ is increasing on $[0,\infty)$ we have 
\[
0 < x'(t) = \int_{[-\tau,0]}\mu(ds)f(x(t+s)) \leq  \int_{[-\tau,0]}\mu(ds)f(x(t)) \leq M f(x(t)), \,\, t > T_1 + \tau.
\]
This means that $x'(t)/x(t) \to 0$ as $t \to \infty$. Notice also that integration of the inequality $x'(s)/f(x(s))\leq M$ 
for $s\in [T_1+\tau,t)$ yields $F(x(t))-F(x(T_1+\tau))\leq M(t-(T_1+\tau))$ for $t\geq \tau$, from which the elementary estimate 
\begin{equation} \label{eq.elemlimsup}
\limsup_{t\to\infty} \frac{x(t)}{F^{-1}(Mt)}\leq 1
\end{equation}
results. In deducing \eqref{eq.elemlimsup}, we have used the fact that the sublinearity of $f$ implies that $F^{-1}(y+c)/F^{-1}(y)\to 1$ as $y\to\infty$ for any $c\in\mathbb{R}$. 

Furthermore, for $t > T_1 + \tau$, $f(x(t+s))\geq f(x(t-\tau))$ for $s \in [-\tau,0]$. Thus
$
x'(t) \geq  M f(x(t-\tau)), \,\, t > T_1 + \tau.
$
Applying the Mean Value Theorem to the continuous function $f \circ x$ for each $t > T_1 + \tau$ there exists $\theta_t \in [0,\tau]$ such that 
$
f(x(t)) = f(x(t-\tau)) + f'(x(t-\theta_t))\tau.
$
Combining this identity with the fact that $f'(x) \to 0$ as $t \to \infty$, we see that $f(x(t-\tau))/f(x(t)) \to 1$ as $t \to \infty$. Hence 
$
\lim_{t \to \infty}x'(t)/f(x(t)) = M.
$
Now for every $\epsilon \in (0,1/2)$ there exists $T_2(\epsilon)>0$ such that 
\[
M(1-\epsilon) < \frac{x'(t)}{f(x(t))} \leq M, \mbox{ for all } t > T_2(\epsilon).
\]
Define next 
\begin{align*}
\tilde{M}(x)&:=\int_{[-\tau,-x]} \mu(ds), \quad x\in [0,\tau].\\
\delta(t)&:=\int_{[-\tau,0]} \mu(ds) \{f(x(t))-f(x(t+s))\},\quad t\geq 0
\end{align*}
For $t\geq \tau$, we have 
\[
\delta(t)=\int_{t-\tau}^t \tilde{M}(t-s)f'(x(s))x'(s)\,ds, \quad t\geq\tau.
\]
Therefore, if we take $T_3(\epsilon)=\max(T_1+\tau,T_2(\epsilon))$ we have 
\[
\delta(t)<\int_{t-\tau}^t \tilde{M}(t-s)f'(x(s)) Mf(x(s))\,ds \leq \int_{t-\tau}^t \tilde{M}(t-s)f'(x(s))\,ds M f(x(t))
\]
and 
\begin{align*}
\delta(t)&>\int_{t-\tau}^t \tilde{M}(t-s)f'(x(s)) M(1-\epsilon) f(x(s))\,ds \\
&\geq \int_{t-\tau}^t \tilde{M}(t-s)f'(x(s))\,ds \cdot M(1-\epsilon) f(x(t-\tau). 
\end{align*}
Since $f(x(t-\tau))/f(x(t))\to 1$ as $t\to\infty$, taking the limit superior and limit inferior as $t\to\infty$, and then letting $\epsilon\to 0^+$ we get $I_1(t)/I(t)\to 1$ as $t\to\infty$, where we have defined 
\[
I_1(t)=\frac{\delta(t)}{f(x(t))M}, \quad 
I(t)=\int_{t-\tau}^t \tilde{M}(t-s)f'(x(s))\,ds.
\]
With this notation, 
\begin{equation} \label{eq.id1}
\frac{1}{M}\frac{x'(t)}{f(x(t))}=1-I_1(t).
\end{equation}
We also define $J$ and $J_1$ by
\begin{equation} \label{def.JJ1}
J(t)=M\int_{T(\epsilon)}^t I(s)\,ds, \quad J_1(t)=M\int_{T(\epsilon)}^t I_1(s)\,ds, \quad t\geq T(\epsilon),
\end{equation}
Next, for every $\epsilon\in (0,1/2)$ define $T(\epsilon)>T_1+\tau$ such that for $t\geq T(\epsilon)$ 
\begin{gather*}
M(1-\epsilon)<\frac{x'(t)}{f(x(t))}\leq M, \quad 
f(x(t-\tau))>(1-\epsilon)f(x(t))
\end{gather*}
Integration of \eqref{eq.id1} over $[T(\epsilon),t]$, and using \eqref{def.JJ1} yields
\begin{equation} \label{eq.FxJ}
F(x(t))-Mt= F(x(T(\epsilon)))-MT(\epsilon)-J_1(t), \quad t\geq T(\epsilon).   
\end{equation}
Next, set 
\[
J^\ast=M\int_{T(\epsilon)-\tau}^{T(\epsilon)} \left(\int_{T(\epsilon)\vee u}^{u+\tau} \tilde{M}(s-u)\,ds\right) f'(x(u))\,du.
\]
We will now prove for $t\geq T(\epsilon)+\tau$, that
\begin{equation} \label{eq.Jrep} 
J(t)=J^\ast+M\int_{T(\epsilon)}^{t-\tau} \int_0^{\tau} \tilde{M}(v)\,dv f'(x(u))\,du + 
M\int_{t-\tau}^t \int_0^{t-u} \tilde{M}(v)\,dv f'(x(u))\,du.
\end{equation}
First, for $t\geq T(\epsilon)+\tau$ we have
\[
J(t)=M\int_{T(\epsilon)}^t I(s)\,ds =  M\int_{T(\epsilon)}^t \int_{s-\tau}^s \tilde{M}(s-u)f'(x(u))\,du \,ds.
\]
By reversing the order of integration we get
\begin{align*}
J(t)&=M\int_{T(\epsilon)-\tau}^t \left(\int_{T(\epsilon)\vee u}^{(u+\tau)\wedge t} \tilde{M}(s-u)\,ds\right) f'(x(u))\,du.
\end{align*}
Splitting the integral gives
\begin{multline*}
J(t)=M\int_{T(\epsilon)-\tau}^{T(\epsilon)} \left(\int_{T(\epsilon)\vee u}^{u+\tau} \tilde{M}(s-u)\,ds\right) f'(x(u))\,du\\
+M\int_{T(\epsilon)}^{t-\tau} \left(\int_{T(\epsilon)\vee u}^{u+\tau} \tilde{M}(s-u)\,ds\right) f'(x(u))\,du
\\+ M\int_{t-\tau}^t \left(\int_{T(\epsilon)\vee u}^{t} \tilde{M}(s-u)\,ds\right) f'(x(u))\,du,
\end{multline*}
and noting that the first integral is $J^\ast$ and tidying up the limits of the integrals yields
\[
J(t)=J^\ast
+M\int_{T(\epsilon)}^{t-\tau} \left(\int_{u}^{u+\tau} \tilde{M}(s-u)\,ds\right) f'(x(u))\,du
+ M\int_{t-\tau}^t \left(\int_{u}^{t} \tilde{M}(s-u)\,ds\right) f'(x(u))\,du.
\]
Substituting $v=s-u$ in the inner integrals now gives \eqref{eq.Jrep}.

Now that we have proven \eqref{eq.Jrep}, we will use it to obtain asymptotic estimates on $J$. 
Since each of the integrands in \eqref{eq.Jrep} are positive for $t\geq T(\epsilon)+\tau$, we have 
\begin{equation} \label{eq.jlower}
J(t)\geq MC\int_{T}^{t-\tau} f'(x(u))\,du, \quad t\geq T(\epsilon)+\tau,
\end{equation}
because 
\[
C=\int_0^{\tau} \tilde{M}(v)\,dv.
\]
We now need a corresponding upper estimate for $J$. Since $\tilde{M}:[0,\tau]\to \mathbb{R}^+$, for $u\in [t-\tau,t]$, we have
\[
\int_{0}^{t-u} \tilde{M}(v)\,dv \leq \int_0^{\tau} \tilde{M}(v)\,dv =C.
\]
 Therefore
\begin{align*}
J(t) &= J^\ast+ MC\int_{T(\epsilon)}^{t-\tau} f'(x(u))\,du +M\int_{t-\tau}^t \int_0^{t-u} \tilde{M}(v)\,dv f'(x(u))\,du \\
&\leq J^\ast+ MC\int_{T(\epsilon)}^{t-\tau} f'(x(u))\,du +M\int_{t-\tau}^t \int_0^{t-u} \tilde{M}(v)\,dv f'(x(u))\,du.
\end{align*}
Thus
\begin{equation} \label{eq.jupper}
J(t) \leq J^\ast+ MC\int_{T(\epsilon)}^t f'(x(u))\,du, \quad t\geq T(\epsilon)+\tau.
\end{equation}
Next, we estimate the integrals on the righthand sides of \eqref{eq.jlower}, \eqref{eq.jupper}.
For $t\geq T(\epsilon)+\tau$ we have 
\begin{align*}
\int_{T(\epsilon)}^{t-\tau} Mf'(x(u))\,du 
&= \int_{T(\epsilon)}^{t-\tau} \frac{f'(x(u))}{f(x(u))}x'(u)\frac{Mf(x(u))}{x'(u)}\,du \\
&\geq \int_{T(\epsilon)}^{t-\tau} \frac{f'(x(u))}{f(x(u))}x'(u)\,du\\
&= \log f(x(t-\tau))-\log f(x(T(\epsilon)))\\
&> \log(1-\epsilon) + \log f(x(t)) - \log f(x(T(\epsilon))).
\end{align*}
Therefore, from \eqref{eq.jlower}, we have 
\[
\liminf_{t\to\infty} \frac{J(t)}{\log f(x(t))}\geq C.
\]
Similarly, we get for $t\geq T(\epsilon)+\tau$ we have 
\begin{align*}
\int_{T(\epsilon)}^{t} Mf'(x(u))\,du 
&= \int_{T(\epsilon)}^{t} \frac{f'(x(u))}{f(x(u))}x'(u)\frac{Mf(x(u))}{x'(u)}\,du \\
&\leq \frac{1}{1-\epsilon}\int_{T(\epsilon)}^{t} \frac{f'(x(u))}{f(x(u))}x'(u)\,du\\
&= \frac{1}{1-\epsilon}\left(\log f(x(t))-\log f(x(T(\epsilon)))\right).
\end{align*}
Therefore, from \eqref{eq.jupper}, we have 
\[
\limsup_{t\to\infty} \frac{J(t)}{\log f(x(t))}\leq C.
\]
Combining this with the limit inferior, we get 
\begin{equation} \label{eq.Jlogf}
\lim_{t\to\infty} \frac{J(t)}{\log f(x(t))}= C.
\end{equation}
Therefore, as we have assumed $f(x)\to\infty$ as $x\to\infty$, we see that $J(t)\to\infty$ as $t\to\infty$. Thus
by \eqref{def.JJ1}, \eqref{eq.Jlogf} and L'H\^opital's rule, we get
\[
\lim_{t\to\infty} \frac{J_1(t)}{\log f(x(t))}= C.
\] 
Putting this limit into \eqref{eq.FxJ} yields \eqref{eq.Fxlogf}. 
The result now follows from Lemma~\ref{lemma.l1}.

\section{Proof of Theorem~\ref{thm.2.3} with Finite First Moment}
Define  $\epsilon_1(t)=\int_{(t,\infty)}\mu(ds)$ for $t\geq 0$ and 
\[
\delta_1(t)=\epsilon_1(t)f(x(t)), \quad t\geq 0.
\]
Clearly $\delta_1(t)>0$ for all $t\geq 0$. Define also $\delta_2$ by 
\[
\delta_2(t)=\int_{[0,t]} \mu(ds)\left(f(x(t))-f(x(t-s))\right), \quad t\geq 0.
\]
We have that $x'(t)\geq 0$ for all $t\geq 0$, and $x(t)\to\infty$ as $t\to\infty$. Therefore there is $T_1^I>0$ such that
$x(t)>x_1$ for all $t\geq T_1^I$. Define $f^\ast=\max_{x\in [0,x_1]} f(x)$. Since $f(x)\to\infty$ as $x\to\infty$, it follows that there is $x_2>x_1$ such that $f(x)>f^\ast$ for all $x\geq x_2$, and there is also $T_1^{II}>0$ such 
that $x(t)>x_2$ for all $t\geq T_1^{II}$. Define $T_1^{III}=\max(T_1^I,T_1^{II})$, and let $t\geq T_1^{III}$. Then 
as $f$ is increasing on $[x_2,\infty)\supset[x_1,\infty)$, we have 
\[
f(x(t))>f(x_2)\geq f^\ast=\max_{y\in [0,x_1]} f(y).
\] 
Now, let $u\in [0,t)$. If $x(u)\leq x_1$, then $f(x(u))\leq f^\ast<f(x(t))$. If $x(u)>x_1$, then $x(t)\geq x(u)>x_1$ and 
$f(x(t))\geq f(x(u))$. Therefore 
\[
f(x(t))>f(x(u)), \quad 0\leq u<t, \quad t\geq T_1^{III}.
\] 
Thus $\delta_2(t)>0$ for all $t\geq T_1^{III}$. Notice for $t\geq 0$ we have 
\[
x'(t)=Mf(x(t))-\delta_1(t)-\delta_2(t).
\]
Since $\delta_1$ and $\delta_2$ are positive on $[T_1^{III},\infty)$, it follows that 
\begin{equation} \label{eq.xprupper}
x'(t)\leq Mf(x(t)), \quad t\geq T_1^{III}.
\end{equation}
Integration leads to 
\begin{equation} \label{eq.2.0}
\limsup_{t\to\infty} \frac{x(t)}{F^{-1}(Mt)}\leq 1.
\end{equation} 
Define for $0\leq a\leq b<+\infty$
\[
M(a,b)=\int_{[a,b]} \mu(ds).
\]
By Fubini's theorem 
\[
\delta_2(t)=\int_0^t M(t-u,t) f'(x(u))x'(u)\,du.
\]
It can be proven, as in the proof of Theorem \ref{thm.2.2}, that $x'(t)/f(x(t))\to M$ as $t\to\infty$. The details are given in \cite[Theorem 1]{applebypatterson2016growth}. From this limit, we have for every $\epsilon\in (0,1)$, that there is $T_1^{IV}(\epsilon)>0$ such that 
\begin{equation} \label{eq.xprlower}
x'(t)>M(1-\epsilon)f(x(t)), \quad t\geq T_1^{IV}(\epsilon).
\end{equation}
Define $T_1(\epsilon)=\max(T_1^{IV}(\epsilon),T_1^{III})$, and finally 
\[
\delta_3(t)=\int_0^{T_1(\epsilon)} M(t-u,t) f'(x(u))x'(u)\,du, \quad t\geq T_1(\epsilon).
\]   
Then for $t\geq T_1(\epsilon)$ we have 
\begin{equation} \label{eq.delt2}
\delta_2(t)=\delta_3(t)+\int_{T_1(\epsilon)}^t M(t-u,t)f'(x(u))x'(u)\,du.
\end{equation}
Also define 
\[
I_1(t)=\frac{1}{M}\epsilon_1(t), \quad \tilde{I}_2(t)=\frac{\delta_2(t)}{Mf(x(t))}. 
\]
Define 
\[
K_1(\epsilon):= \int_0^{T_1(\epsilon)} |f'(x(u))|x'(u)\,du.
\]
Then for $t\geq T_1(\epsilon)$, we have 
\begin{equation} \label{eq.delt3bdd}
|\delta_3(t)|\leq K_1(\epsilon) \int_{[t-T_1(\epsilon),t]} \mu(ds)=: \delta_4(t).
\end{equation}
Since $t\mapsto f(x(t))$ is increasing on $[T_1,\infty)$, we get from \eqref{eq.xprupper}, \eqref{eq.delt2}, and \eqref{eq.delt3bdd} the bound
\[
\delta_2(t)\leq \delta_4(t) + M\int_{T_1(\epsilon)}^t M(t-u,t)f'(x(u))\,du \cdot f(x(t)).
\]
Since 
\[
\lim_{t\to\infty} \int_{[0,t]} s\mu(ds)=C\in (0,\infty),
\]
it follows for every $\epsilon\in (0,1)$ that there exists $T_2(\epsilon)>0$ such that 
\[
\int_{[0,T_2(\epsilon)]} s\mu(ds)\geq C(1-\epsilon).
\]
We also have that 
\[
\lim_{t\to\infty} \frac{f(x(t-T_2(\epsilon)))}{f(x(t))}=1.
\]
Therefore, for every $\eta\in (0,1)$ there is $T_3'(\eta,\epsilon)>0$ such that for all $t\geq T_3'(\eta,\epsilon)$ 
we have $f(x(t-T_2(\epsilon)))>(1-\eta)f(x(t))$. Fix $\eta=\epsilon$ and set $T_3'(\epsilon)=T_3'(\epsilon,\epsilon)$. 
Then for $t\geq T_3'(\epsilon)$ we have $f(x(t-T_2(\epsilon)))>(1-\epsilon)f(x(t))$. Now, let $t\geq T_1(\epsilon)+T_2(\epsilon)+T_3'(\epsilon)$. Then from \eqref{eq.xprlower}, \eqref{eq.delt2} and \eqref{eq.delt3bdd} we have 
\begin{align*}
\delta_2(t)&\geq -|\delta_3(t)|+\int_{t-T_2(\epsilon)}^t M(t-u,t)f'(x(u))x'(u)\,du \\
&>-\delta_4(t)+\int_{t-T_2(\epsilon)}^t M(t-u,t)f'(x(u)) M(1-\epsilon)f(x(u))\,du \\
&>-\delta_4(t)+M(1-\epsilon) \int_{t-T_2(\epsilon)}^t M(t-u,t)f'(x(u)) \,du \cdot f(x(t-T_2(\epsilon)))\\
&>-\delta_4(t)+M(1-\epsilon)^2 \int_{t-T_2(\epsilon)}^t M(t-u,t)f'(x(u)) \,du \cdot f(x(t))).
\end{align*}
Define 
\[
\tilde{I}_3(t)=\frac{\delta_4(t)}{Mf(x(t))}>0, \quad t\geq T_1(\epsilon)+T_2(\epsilon).
\]
Then for $t\geq T_1(\epsilon)+T_2(\epsilon)+T_3'(\epsilon)=:T_3(\epsilon)$, we have 
\begin{multline} \label{eq.2.5}
-\tilde{I}_3(t)+(1-\epsilon)^2 \int_{t-T_2(\epsilon)}^t M(t-u,t)f'(x(u)) \,du
<
\tilde{I}_2(t)\\
<\tilde{I}_3(t)+ \int_{T_1(\epsilon)}^t M(t-u,t)f'(x(u))\,du.
\end{multline}
Since $x'(t)/(Mf(x(t)))=1-I_1(t)-\tilde{I}_2(t)$, by defining 
\[
J(t)=\int_{T_3(\epsilon)}^t M\tilde{I}_2(s), \quad t\geq T_3(\epsilon)
\]
integration yields
\begin{equation} \label{eq.2.6}
F(x(t))-Mt=F(x(T_3(\epsilon)))-MT_3(\epsilon) - \int_{T_3(\epsilon)}^t \epsilon_1(s)\,ds - J(t), \quad t\geq T_3(\epsilon).
\end{equation}
We can readily estimate the third term on the right--hand side: for $t\geq T_3(\epsilon)$ we have by Fubini's theorem
\begin{align*}
 \int_{T_3(\epsilon)}^t \epsilon_1(s)\,ds
&=\int_{[T_3(\epsilon),\infty)} \int_{[T_3(\epsilon),t\wedge u]} \, ds \,\mu(du)\\
&=\int_{[T_3(\epsilon),\infty)} (t\wedge u - T_3(\epsilon))\, \mu(du)
\leq \int_{[T_3(\epsilon),\infty)} ( u - T_3(\epsilon))\, \mu(du)\leq C.
\end{align*}
We estimate for $t\geq T_3(\epsilon)$ the integral
\[
\int_{T_3(\epsilon)}^t M\tilde{I}_3(s)\,ds.
\]
Since $f$ and $x$ are increasing, by \eqref{eq.delt3bdd} and Fubini's theorem we get 
\begin{align*}
\int_{T_3(\epsilon)}^t M\tilde{I}_3(s)\,ds
&\leq \frac{K_1(\epsilon)}{f(x(T_3(\epsilon)))}\int_{T_3(\epsilon)}^t \int_{[s-T_1(\epsilon),s]}\mu(du)\,ds\\
&\leq \frac{K_1(\epsilon)}{f(x(T_3(\epsilon)))}\int_{T_3(\epsilon)}^\infty \int_{[s-T_1(\epsilon),s]}\mu(du)\,ds\\
&=\frac{K_1(\epsilon)}{f(x(T_3(\epsilon)))} \int_{[T_3(\epsilon)-T_1(\epsilon),\infty)} (u+T_1(\epsilon)-T_3(\epsilon))\,\mu(du)=:C_1(\epsilon).
\end{align*}
Therefore 
\begin{equation} \label{eq.2.7}
 0\leq \int_{T_3(\epsilon)}^t \epsilon_1(s)\,ds \leq C, \quad 
0\leq \int_{T_3(\epsilon)}^t M\tilde{I}_3(s)\,ds \leq C_1(\epsilon), \quad t\geq T_3(\epsilon).
\end{equation}
From the definition of $J$, \eqref{eq.2.5} and \eqref{eq.2.7}, for $t\geq T_3(\epsilon)$ we have 
\begin{align} \label{eq.2.81}
J(t)&\geq -C_1(\epsilon)+M(1-\epsilon)^2 \int_{T_3(\epsilon)}^t \int_{s-T_2(\epsilon)}^s M(s-u,s)f'(x(u)) \,du\,ds, \\
\label{eq.2.82}
J(t)&\leq C_1(\epsilon)+ M\int_{T_3(\epsilon)}^t \int_{T_1(\epsilon)}^s M(s-u,s)f'(x(u))\,du\,ds.
\end{align}
Next, set $T_4(\epsilon)=T_2(\epsilon)+T_3(\epsilon)$, and let $t\geq T_4(\epsilon)$. By reversing the order of integration in \eqref{eq.281} and splitting the integral, and using the positivity of the integrands, we get
\begin{align*}
J(t)&\geq -C_1(\epsilon)+
M(1-\epsilon)^2 \int_{T_3(\epsilon)-T_2(\epsilon)}^{T_3(\epsilon)} \int_{T_3(\epsilon)\vee u}^{t\wedge(u+T_2)} 
M(s-u,s)\,ds f'(x(u)) \,du\\
&\qquad+ 
M(1-\epsilon)^2 \int_{T_3(\epsilon)}^{t} \int_{u}^{t\wedge(u+T_2(\epsilon))} 
M(s-u,s)\,ds f'(x(u)) \,du\\
&> -C_1(\epsilon)+
M(1-\epsilon)^2 \int_{T_3(\epsilon)}^{t-T_2(\epsilon)} \int_{u}^{t\wedge(u+T_2(\epsilon))} 
M(s-u,s)\,ds f'(x(u)) \,du\\
&\qquad+ 
M(1-\epsilon)^2 \int_{t-T_2(\epsilon)}^{t} \int_{u}^{t} M(s-u,s)\,ds f'(x(u)) \,du\\
&> -C_1(\epsilon)+
M(1-\epsilon)^2 \int_{T_3(\epsilon)}^{t-T_2(\epsilon)} \int_{u}^{u+T_2(\epsilon)} M(s-u,s)\,ds f'(x(u)) \,du.
\end{align*} 
For $u\in [T_3,t-T_2]$, by making the substitution $v=s-u$ and reversing the order of integration we get
\begin{align*}
\int_{u}^{u+T_2(\epsilon)} M(s-u,s)\,ds
&=\int_{0}^{T_2(\epsilon)} M(v,v+u)\,dv=\int_{0}^{T_2(\epsilon)} \int_{[v,v+u]} \mu(dw)\,dv\\
&=\int_{[0,T_2(\epsilon)+u]} \left( w\wedge T_2(\epsilon) - (w-u)\vee 0\right) \mu(dw)\\
&=\int_{[0,T_2(\epsilon)]} w\mu(dw) \\
&\qquad+
\int_{(T_2(\epsilon),T_2(\epsilon)+u]} \left(T_2(\epsilon) - (w-u)\vee 0\right) \mu(dw).
\end{align*}
Since the integrand in the second integral is non--negative, we have by the definition of $T_2$,
\[
\int_{u}^{u+T_2(\epsilon)} M(s-u,s)\,ds \geq \int_{[0,T_2(\epsilon)]} w\mu(dw) \geq C(1-\epsilon).
\] 
Therefore for $t\geq T_4(\epsilon)$ we have 
\begin{equation} \label{eq.2.91}
J(t)>-C_1(\epsilon)+
MC(1-\epsilon)^3 \int_{T_3(\epsilon)}^{t-T_2(\epsilon)}  f'(x(u)) \,du.
\end{equation}
For $t\geq T_4(\epsilon)$, because $T_3>T_1$ we have from \eqref{eq.2.82} and an interchange of integration order
\begin{align*}
J(t)&\leq C_1(\epsilon)+ M\int_{T_3(\epsilon)}^t \int_{T_1(\epsilon)}^s M(s-u,s)f'(x(u))\,du\,ds\\
&\leq C_1(\epsilon)+ M\int_{T_1(\epsilon)}^t \int_{T_3(\epsilon)\vee u}^t M(s-u,s)\,ds f'(x(u))\,du.
\end{align*}
Splitting the integral gives for $t\geq T_4(\epsilon)$
\begin{multline} \label{eq.2.10}
J(t)\leq C_1(\epsilon)+ M\int_{T_1(\epsilon)}^{T_3(\epsilon)} \int_{T_3(\epsilon)}^t M(s-u,s)\,ds f'(x(u))\,du\\
+ M\int_{T_3(\epsilon)}^t \int_{u}^t M(s-u,s)\,ds f'(x(u))\,du.
\end{multline}
It can now be checked that 
\begin{equation} \label{eq.2.111}
\int_u^t M(s-u,s)\,ds \leq \int_{[0,t]} w\mu(dw), \quad t\geq 2u, t\geq u\geq T_3(\epsilon),
\end{equation}
and likewise that 
\begin{equation} \label{eq.2.112}
\int_u^t M(s-u,s)\,ds \leq \int_{[0,t]} w\mu(dw), \quad t<2u, t\geq u\geq T_3(\epsilon).
\end{equation}
We defer the proof of these estimates to the end. 
Putting \eqref{eq.2.111} and \eqref{eq.2.112} into \eqref{eq.2.10} yields for $t\geq T_4(\epsilon)$
\begin{equation} \label{eq.2.12}
J(t)\leq C_1(\epsilon)+ M C\int_{T_3(\epsilon)}^t f'(x(u))\,du
+ M\int_{T_1(\epsilon)}^{T_3(\epsilon)} \int_{T_3}^t M(s-u,s)\,ds f'(x(u))\,du.
\end{equation}
Next for $u\in [T_1,T_3]$ and $t\geq T_4$, we get, by making the substitution $v=s-u$, and an exchange of order of integration 
\begin{align*}
\int_{T_3}^t M(s-u,s)\,ds
&=\int_{T_3-u}^{t-u} \int_{[v,v+u]} \mu(dw)\,dv\\
&\leq \int_{0}^{t-u} \int_{[v,v+u]} \mu(dw)\,dv\\
&=\int_{[0,u]} \left( (t-u)\wedge w\right)\,\mu(dw) + \int_{(u,t]}  \left( (t-u)\wedge w - (w-u)\right)  \mu(dw).
\end{align*} 
Again, considering the cases $t\geq 2u$ and $t<2u$, we arrive at the estimates
\begin{equation} \label{eq.2.131}
\int_{T_3}^t M(s-u,s)\,ds \leq \int_{[0,t]} w\mu(dw), \quad t\geq 2u, t\geq T_4(\epsilon), u\in [T_1,T_3],
\end{equation}
and 
\begin{equation} \label{eq.2.132}
\int_{T_3}^t M(s-u,s)\,ds \leq \int_{[0,t]} w\mu(dw), \quad t<2u, t\geq T_4(\epsilon), u\in [T_1,T_3].
\end{equation}
We postpone the justification of these inequalities to the end. 
Using the fact that $\int_{[0,t]} w\mu(dw)\leq C$ for all $t\geq 0$, and 
putting \eqref{eq.2.131} and \eqref{eq.2.132} into \eqref{eq.2.12}, yields 
\begin{equation} \label{eq.2.14}
J(t)\leq C_1(\epsilon)+ M C\int_{T_1(\epsilon)}^t f'(x(u))\,du, \quad t\geq T_4(\epsilon)
\end{equation}

Next for $t\geq T_4$ we estimate the integral in \eqref{eq.2.91}: using \eqref{eq.xprupper} and the fact that 
for $t\geq T_3'(\epsilon)$ we have $f(x(t-T_2(\epsilon)))>(1-\epsilon)f(x(t))$, we get
\begin{align*}
M \int_{T_3(\epsilon)}^{t-T_2(\epsilon)}  f'(x(u)) \,du
&= \int_{T_3(\epsilon)}^{t-T_2(\epsilon)}  \frac{f'(x(u))}{f(x(u))} \frac{Mf(x(u))}{x'(u)}x'(u)\,du\\
&\geq \int_{T_3(\epsilon)}^{t-T_2(\epsilon)}  \frac{f'(x(u))}{f(x(u))} x'(u)\,du\\
&=\log f(x(t-T_2(\epsilon)))-\log f(x(T_3(\epsilon)))\\
&>\log f(x(t))+\log (1-\epsilon) -\log f(x(T_3(\epsilon))). 
\end{align*}
Therefore from \eqref{eq.2.91}, we get 
\[
\liminf_{t\to\infty} \frac{J(t)}{\log f(x(t))}\geq C(1-\epsilon)^3.
\]
Letting $\epsilon\to 0^+$ yields
\begin{equation} \label{eq.2.15}
\liminf_{t\to\infty} \frac{J(t)}{\log f(x(t))}\geq C.
\end{equation}
For $t\geq T_4(\epsilon)$, we estimate the integral in \eqref{eq.2.14}. Using \eqref{eq.xprlower} we get
\begin{align*} 
J(t)&\leq C_1(\epsilon)+ M C\int_{T_1(\epsilon)}^t f'(x(u))\,du \\
&= C_1(\epsilon)+ C\int_{T_1(\epsilon)}^t \frac{f'(x(u))}{f(x(u))}\cdot \frac{Mf(x(u))}{x'(u)}x'(u)\,du \\
&\leq  C_1(\epsilon)+ \frac{C}{1-\epsilon}\int_{T_1(\epsilon)}^t \frac{f'(x(u))}{f(x(u))}x'(u)\,du \\
&=  C_1(\epsilon)+ \frac{C}{1-\epsilon}\left( \log f(x(t)) - \log f(x(T_1(\epsilon)))\right).
\end{align*}
Dividing across by $\log f(x(t))$, taking the limsup as $t\to\infty$, and then letting $\epsilon\to 0^+$ 
yields
\begin{equation*}
\limsup_{t\to\infty} \frac{J(t)}{\log f(x(t))}\leq C.
\end{equation*}
Combining this with \eqref{eq.2.15} gives
\begin{equation} \label{eq.2.17}
\lim_{t\to\infty} \frac{J(t)}{\log f(x(t))} = C.
\end{equation}
For $t\geq T_3(\epsilon)$, by \eqref{eq.2.6}, we have 
\[
\frac{F(x(t))-Mt}{\log f(x(t))}=\frac{F(x(T_3(\epsilon)))-MT_3(\epsilon) 
- \int_{T_3(\epsilon)}^t \epsilon_1(s)\,ds}{\log f(x(t))} - \frac{J(t)}{\log f(x(t))}.
\]
Since $0\leq \int_{T_3(\epsilon)}^t \epsilon_1(s)\,ds \leq C$, $\log f(x(t))\to \infty$ as $t\to\infty$ 
and \eqref{eq.2.17} holds, we immediately get 
\begin{equation} \label{eq.2.18}
\lim_{t\to\infty} \frac{F(x(t))-Mt}{\log f(x(t))}=-C.
\end{equation}
Recall that $x$ obeys \eqref{eq.2.0}, and $f$ obeys \eqref{eq.lambda} with $\lambda\in [0,\infty]$.
Therefore, we may apply Lemma~\ref{lemma.l1} to $x$ obeying \eqref{eq.2.0} and \eqref{eq.2.18}, 
from which we conclude that 
\[
\lim_{t\to\infty} \frac{x(t)}{F^{-1}(Mt)}=e^{-\lambda C},
\]
as required. This completes the proof of Theorem~\ref{thm.2.3} when $C<+\infty$.

It remains to dispense with the estimates \eqref{eq.2.111} and \eqref{eq.2.112}, 
as well as \eqref{eq.2.131} and \eqref{eq.2.132}. 
We start with \eqref{eq.2.111} and \eqref{eq.2.112}. For $t\geq u\geq T_3(\epsilon)$ we have 
\begin{align}
\int_u^t M(s-u,s)\,ds 
&= \int_{0}^{t-u} \int_{[v,v+u]} \mu(dw)\,dv \nonumber\\
&=\int_{[0,t]} \{(t-u)\wedge w - (w-u)\vee 0\} \mu(dw) \nonumber\\
&=\int_{[0,u)} \{(t-u)\wedge w \} \mu(dw) 
+ 
\int_{[u,t]}  \{(t-u)\wedge w - (w-u)\} \mu(dw).
\label{eq.keyM}
\end{align}
We now use \eqref{eq.keyM} to prove \eqref{eq.2.111} and \eqref{eq.2.112}.

If $t\geq 2u$, $t-u\geq u$, so 
\[
\int_{[0,u)} \{(t-u)\wedge w \} \mu(dw)=\int_{[0,u)} w\mu(dw).
\]
Similarly
\begin{align*}
\lefteqn{\int_{[u,t]}  \{(t-u)\wedge w - (w-u)\} \mu(dw)}\\
&=
\int_{[u,t-u]}  \{(t-u)\wedge w - (w-u)\} \mu(dw)
+
\int_{(t-u,t]}  \{(t-u)\wedge w - (w-u)\} \mu(dw)\\
&=
\int_{[u,t-u]}  u\mu(dw)+ \int_{(t-u,t]}  (t-w)  \mu(dw).
\end{align*}
Since $w\geq u$ in the first integral, and $w\geq t-u$ and $w-u\geq t-2u\geq 0$ in the second, we have 
\[
\int_{[u,t]}  \{(t-u)\wedge w - (w-u)\} \mu(dw) \leq \int_{[u,t]} w\mu(dw).
\] 
Combining this with the expression we have for the integral on $[0,u)$ in \eqref{eq.keyM} now gives 
the estimate in \eqref{eq.2.111}.

Now suppose that $t<2u$ so $t-u<u$. Then the first integral in \eqref{eq.keyM} is
\[
\int_{[0,u)} \{(t-u)\wedge w \} \mu(dw) \leq \int_{[0,u)} u\wedge w \mu(dw) = \int_{[0,u)} w\mu(dw).
\]
For $w\in [u,t]$, $t<2u$ we have $t-w\leq t-u<u\leq w$, it follows that 
\[
\int_{[u,t]}  \{(t-u)\wedge w - (w-u)\} \mu(dw) \leq \int_{[u,t]} w\mu(dw).
\]
Combining this with the first identity in this paragraph gives \eqref{eq.2.112}.

Now we turn to the proof of \eqref{eq.2.131} and \eqref{eq.2.132}: for $u\in [T_1,T_3]$ and $t\geq T_4$, we get 
\begin{align*}
\int_{T_3}^t M(s-u,s)\,ds &= \int_{T_3-u}^{t-u} \int_{[v,v+u]} \mu(dw) \,dv
\leq \int_0^{t-u} \int_{[v,v+u]}\mu(dw) \,dv.
\end{align*}
Hence 
\begin{equation} \label{def.M1}
\int_{T_3}^t M(s-u,s)\,ds\leq \int_0^{t-u} \int_{[w,w+u]} \mu(dv) \,dw=:M_1(u,t).
\end{equation}
Now for $t\geq u$ we have
\begin{align*}
M_1(u,t)
=\int_{[0,t]} \int_{(v-u)\vee 0}^{v\wedge(t-u)} \,dw \mu(dv)
=\int_{[0,t]} \{v\wedge(t-u) - (v-u)\vee 0\} \mu(dv).
\end{align*}
If $t>2u$ we have
\begin{align*}
M_1(u,t)
&=\int_{[0,u)} v  \mu(dv)
+   \int_{[u,t-u)} u \mu(dv) 
+ \int_{[t-u,t]} (t - v)\mu(dv)\\
&\leq \int_{[0,u)} v  \mu(dv)
+   \int_{[u,t-u)} v \mu(dv) 
+ \int_{[t-u,t]} (t - v)\mu(dv).
\end{align*}
In the last integrand $v\geq t-u>u$, so $t-v\leq u<v$. Hence
\[
M_1(u,t)\leq \int_{[0,t]} v  \mu(dv), \quad t>2u.
\]
If $t\leq 2u$ we have 
\begin{align*}
M_1(u,t)&=\int_{[0,t-u)} v   \mu(dv)
+   \int_{[t-u,u)} (t-u)    \mu(dv) 
+ \int_{[u,t]} (t-v)  \mu(dv)\\
&\leq \int_{[0,t-u)} v  \mu(dv)
+   \int_{[t-u,u)} v \mu(dv) 
+ \int_{[u,t]} (t - v)\mu(dv).
\end{align*}
In the last integrand we have $t\geq v\geq u$, so $t-v\leq t-u\leq u\leq v$. Therefore
\[
M_1(u,t)
\leq  \int_{[0,t]} v  \mu(dv), \quad t\leq 2u.
\]
Combining the cases where $t>2u$ and $t\leq 2u$ we have the consolidated estimate
\begin{equation} \label{eq.estM1}
M_1(u,t)\leq  \int_{[0,t]} v  \mu(dv), \quad t\geq u.
\end{equation}
Thus
\[
\int_{T_3}^t M(s-u,s)\,ds \leq \int_{[0,t]} v  \mu(dv), \quad t\geq u\geq T_3.
\]
establishing both \eqref{eq.2.131} and \eqref{eq.2.132}. This completes the proof.

\section{Proof of Theorem~\ref{thm.2.3} with Infinite First Moment}
By the same considerations made in the case when $C<+\infty$, we have 
\[
x'(t)\leq Mf(x(t)), \quad t\geq T_1^{III}, \quad x'(t)>\frac{M}{2} f(x(t)), \quad t\geq T^{IV}(1/2), 
\]
and \eqref{eq.2.0} holds. We take $T_1=\max(T_1^{III},T_1^{IV})$ recalling the definition of $T_1^{III}$ 
in the case when $C<+\infty$. For $t\geq T_1$, we still have the estimate 
\[
|\delta_3(t)|\leq \int_{[t-T_1,t]} \mu(ds) \cdot K_1=:\delta_4(t)
\] 
where 
\[
K_1=\int_0^{T_1} |f'(x(u))|x'(u)\,du.
\]
Next, as $\int_{[0,t]} s\mu(ds)\to\infty$ as $t\to\infty$, for every $N\in\mathbb{N}$ there is $T_2=T_2(N)$ 
such that 
\begin{equation} \label{eq.3.2.1}
\int_{[0,T_2(N)]} s\mu(ds) >N.
\end{equation} 
Since $T_2(N)$ is fixed, the limit 
\[
\lim_{t\to\infty} \frac{f(x(t-T_2(N)))}{f(x(t))}=1
\]
prevails. Therefore, for every $\eta\in (0,1)$ there is $\tilde{T}_3(\eta,N)>0$ such that 
$t\geq \tilde{T}_3(\eta,N)$ implies $f(x(t-T_2(N)))>(1-\eta)f(x(t))$. Set $\eta=1/2$. Then, with 
$T_3'(N)= \tilde{T}_3(1/2,N)$, we have 
\[
f(x(t-T_2(N)))> \frac{1}{2}f(x(t)), \quad t\geq T_3'(N).
\]
Hence, for $t\geq T_1+T_2(N)+T_3'(N)$, we can argue as above to obtain 
\[
\tilde{I}_2(t)\geq -\tilde{I}_3(t)+\frac{M}{4}\int_{t-T_2}^t M(t-u,t) f'(x(u))\,du,
\] 
where $\tilde{I}_2(t)=\delta_2(t)/(Mf(x(t)))$, $\tilde{I}_3(t)=\delta_4(t)/(Mf(x(t)))$. Define 
$T_3(N)=T_1+T_2(N)+T_3'(N)$. For $t\geq T_3(N)$ we have 
\begin{align*}
F(x(t))-Mt &= F(x(T_3)) - MT_3 - \int_{T_3}^t \epsilon_1(s)\,ds  - \int_{T_3}^t M\tilde{I}_2(s)\,ds
\end{align*}
Hence for $t\geq T_3(N)$ we have
\begin{equation} \label{eq.3.2.2}
F(x(t))-Mt 
\leq F(x(T_3)) - MT_3 + M\int_{T_3}^t \tilde{I}_3(s)\,ds 
- \frac{M}{4} \int_{T_3}^t \int_{s-T_2}^s M(s-u,s) f'(x(u))\,du.
\end{equation}
Next, we estimate the third term on the righthand side of \eqref{eq.3.2.2}. By definition for $t\geq T_3$, 
we get 
\[
M\int_{T_3}^t \tilde{I}_3(s)\,ds = K_1\int_{T_3}^t \frac{1}{f(x(s))} \int_{[s-T_1,s]} \mu(du)\,ds
\leq K_1 M \int_{T_3}^t \frac{1}{f(x(s))}\,ds.
\]
Since $t\geq T_3>T_1^{IV}$ we have 
\begin{align*}
M\int_{T_3}^t \tilde{I}_3(s)\,ds 
&\leq K_1 \int_{T_3}^t \frac{x'(s)}{f^2(x(s))} \cdot \frac{Mf(x(s))}{x'(s)}\,ds \\
&\leq 2K_1 \int_{T_3}^t \frac{x'(s)}{f^2(x(s))}\,ds
=2K_1 \int_{x(T_3)}^{x(t)} \frac{1}{f^2(u)}\,du.
\end{align*}
Now, as $\lim_{x\to\infty} f(x)/(x/\log x)=\lambda\in (0,\infty]$ and $f(x)/x\to 0$ as $x\to\infty$, it 
follows that $\log f(x)/\log x\to 1$ as $x\to\infty$. Hence 
\[
\lim_{x\to\infty} \frac{\log (1/f^2(x))}{\log x} = -2.
\]
Therefore $\int_1^\infty f^{-2}(u)\,du<+\infty$, and so as $x(T_3)>x_1$ we have 
\begin{equation} \label{eq.3.2.3}
M\int_{T_3}^t \tilde{I}_3(s)\,ds \leq 2K_1 \int_{x_1}^\infty \frac{1}{f^2(u)}\,du, \quad t\geq T_3.
\end{equation}
Letting 
\[
K_2(N)=F(x(T_3(N)))-MT_3(N)+2K_1 \int_{x_1}^\infty \frac{1}{f^2(u)}\,du,
\]
we have from \eqref{eq.3.2.3} and \eqref{eq.3.2.2} that
\begin{equation} \label{eq.3.2.4}
F(x(t))-Mt \leq K_2(N)-\frac{M}{4} \int_{T_3}^t \int_{s-T_2}^s M(s-u,s) f'(x(u))\,du, \quad t\geq T_3(N).
\end{equation}
Let $T_4(N)=T_2(N)+T_3(N)$ and $t\geq T_4(N)$. We estimate the second term on the righthand side of 
\eqref{eq.3.2.4} as in the proof of the lower bound of $J$ in Theorem~\ref{thm.2.3} after \eqref{eq.2.81}. 
Noting that $f'(x(u))>0$ for all $u\geq T_3-T_2$, for $t\geq T_4(N)$ we get
\begin{align*}
\lefteqn{
M\int_{T_3}^t \int_{s-T_2}^s M(s-u,s) f'(x(u))\,du\,ds}\\
&=M\int_{T_3-T_2}^{T_3} \int_{T_3\vee u}^{t\wedge(u+T_2)} M(s-u,s)\,ds f'(x(u))\,du
\\
&\qquad+M\int_{T_3}^{t-T_2} \int_{u}^{u+T_2} M(s-u,s)\,ds f'(x(u))\,du\\
&\qquad\qquad+M\int_{t-T_2}^t \int_{u}^{u+T_2} M(s-u,s)\,ds f'(x(u))\,du\\
&>M\int_{T_3}^{t-T_2} \int_{u}^{u+T_2} M(s-u,s)\,ds f'(x(u))\,du.
\end{align*}
For $u\in [T_3,t-T_2]$ we have as before that 
\[
\int_{u}^{u+T_2} M(s-u,s)\,ds \geq \int_{[0,T_2]} w\mu(dw) >N.
\]
Therefore from \eqref{eq.3.2.4} for $t\geq T_4(N)$ we have 
\begin{equation} \label{eq.3.2.5}
F(x(t))-Mt \leq K_2(N) - \frac{MN}{4} \int_{T_3}^{t-T_2} f'(x(u))\,du.
\end{equation}
Finally, for $t\geq T_4(N)$ we get
\begin{align*}
M\int_{T_3}^{t-T_2} f'(x(u))\,du  
&= \int_{T_3}^{t-T_2} \frac{f'(x(u))}{f(x(u))}\cdot \frac{Mf(x(u)))}{x'(u)} x'(u)\,du\\
&\geq \int_{T_3}^{t-T_2} \frac{f'(x(u))}{f(x(u))} x'(u)\,du\\
&=\log f(x(t-T_2)) - \log f(x(T_3))\\
&> \log\left(\frac{1}{2}\right) + \log f(x(t))  - \log f(x(T_3)).
\end{align*}
Since $f(x(t))\to\infty$ as $t\to\infty$, taking this estimate together with \eqref{eq.3.2.5} and letting 
$t\to\infty$, we get 
\[
\liminf_{t\to\infty} \frac{F(x(t))-Mt}{\log f(x(t))} \leq - \frac{N}{4}.
\] 
Since $N$ is arbitrary, we get 
\[
\lim_{t\to\infty} \frac{F(x(t))-Mt}{\log f(x(t))} =-\infty,
\]
and because $\log f(x)/\log x\to 1$ as $x\to\infty$, we have
\begin{equation*} 
\lim_{t\to\infty} \frac{F(x(t))-Mt}{\log x(t)} =-\infty.
\end{equation*}
Notice that the estimate $x'(t)\leq Mf(x(t))$ for $t\geq T_1^{III}$ holds, so asymptotic integration
yields 
\[
\limsup_{t\to\infty} \frac{x(t)}{F^{-1}(Mt)}\leq 1.
\]
Therefore all the hypotheses of Lemma~\ref{lemma.l8} hold, and therefore $x(t)/F^{-1}(Mt)\to 0$ as $t\to\infty$, 
as claimed.

\section{Proof of Theorems~\ref{thm.2.2.1.1},~\ref{thm.2.2.1}, and~\ref{thm.2.4}}  
The proofs of these results rely upon some preliminary lemmas. The first several results will be employed in the proof of 
Theorems~\ref{thm.2.2.1} and \ref{thm.2.2.1.1}, although Lemma~\ref{lemma.lemmaa1} is also needed for the proof of 
Theorem~\ref{thm.2.4}.

\begin{lemma} \label{lemma.fprdown}
Suppose that $f(x)>0$ for all $x>0$, $f'(x)>0$ for all $x> x_1$, $f'(x) \to 0$ as $x \to \infty$ and 
$f(x)\to\infty$ as $x\to\infty$. If $f'$ is decreasing on $[x_2,\infty)$, then 
\begin{align} \label{eq.fxdecalmost}
&\text{For every $\epsilon>0$ there is $x_0(\epsilon)>0$ such that $x>y\geq x_0(\epsilon)$ implies }\nonumber\\
& \hspace{100pt}\frac{f(x)}{x}< (1+\epsilon)\frac{f(y)}{y}.
\end{align}
\end{lemma}
\begin{proof}
Let $u>\max(x_2,x_1)=:x_3$. Since $f'$ is decreasing, we have
\[
f(u)-f(x_3)\geq f'(u)(u-x_3).
\] 
Rearranging and integrating over the interval $[y,x]$ (for $x>x_3$) yields
\[
\frac{f(x)-f(x_3)}{x-x_3}\leq \frac{f(y)-f(x_3)}{y-x_3}.
\]
Define 
\[
\alpha(x):= \left(\frac{f(x)-f(x_3)}{x-x_3}\right)\bigg/\left(\frac{f(x)}{x}\right), \quad x>x_3.
\]
Then 
\[
\frac{f(x)}{x}\leq \frac{\alpha(y)}{\alpha(x)}\cdot \frac{f(y)}{y}, \quad x>y>x_3.
\]
Since $f(x)\to\infty$ as $x\to\infty$, it follows that $\alpha(x)\to 1$ as $x\to\infty$. Therefore, for every $\epsilon>0$ there is $x_4(\epsilon)>0$ such that 
\[
\frac{1}{\sqrt{1+\epsilon}}<\alpha(x)<\sqrt{1+\epsilon}, \quad x>x_4(\epsilon).
\]
Now, set $x_0(\epsilon)=\max(x_3+1,x_4(\epsilon))$. Then for $x>y\geq x_0(\epsilon)$ we have \eqref{eq.fxdecalmost} 
as claimed. 
\end{proof}

The following result, which was established in \cite{applebypatterson2016growth} for increasing, concave functions, will also be used. Scrutiny of the proof in \cite{applebypatterson2016growth} shows that the monotonicity restrictions can be relaxed to the ultimate monotonicity hypotheses imposed here.

\begin{lemma}\label{asym_equiv_copy_ap}
Suppose $\varphi$ is such that $\varphi(x)\to\infty$ as $x\to\infty$, $\varphi'(x)>0$ for $x>x_1$ and $\varphi'(x)$ is decreasing on $[x_2,\infty)$ with $\varphi'(x)\to 0$ as $x\to\infty$. If $b,c \in C(\mathbb{R^+},\mathbb{R^+})$ obey 
$\lim_{t\to\infty}b(t)=\lim_{t\to\infty}c(t) = \infty$, and $b(t) \sim c(t)$ as $t\to\infty$, then $\varphi(b(t)) \sim \varphi(c(t))$ as $t\to\infty$.
\end{lemma}

\begin{lemma}\label{lemma.lemmaa1}
Let $M>0$. Suppose that $f(x)>0$ for all $x>0$, $f'(x)>0$ for all $x > x_1$, $f'(x) \to 0$ as $x \to \infty$. Define $F$ as in \eqref{def.F}. Suppose that $a$ is a measurable function such that $a(t)>0$ for all $t\geq T^\ast$. 
\begin{itemize}
\item[(a)]If 
\begin{equation} \label{eq.asmall}
\lim_{x\to\infty} \frac{f(x)}{x}a\left(F(x)/M\right)=0,
\end{equation}
then 
\begin{equation} \label{eq.Fmin1aFmin1}
\lim_{t\to\infty} \frac{F^{-1}(Mt-a(t))}{F^{-1}(Mt)}=1.
\end{equation}
\item[(b)]
If $f'$ is decreasing on $[x_2,\infty)$, and $f(x)\to\infty$ as $x\to\infty$, then \eqref{eq.Fmin1aFmin1} implies \eqref{eq.asmall}.
\end{itemize}
\end{lemma}
\begin{proof}
We start by proving that \eqref{eq.asmall} implies \eqref{eq.Fmin1aFmin1}.
Since $f$ is increasing, for $x\geq x_1$ we have 
\[
F(x)-F(x_1)=\int_{x_1}^x \frac{1}{f(u)} \geq  \frac{x-x_1}{f(x)}.
\]
Thus 
\[
\liminf_{x\to\infty} \frac{F(x)f(x)}{x} \geq 1.
\]
Thus for every $\epsilon\in (0,1)$ there is $x_0(\epsilon)>0$ such that 
$f(x)/x>(1-\epsilon)/F(x)$ for $x\geq x_0(\epsilon)$. Now, since $F(x)\to\infty$ as $x\to\infty$, we have that $F(x)/M>T^\ast$ for all $x>x_2$. Let $x_3(\epsilon)=\max(x_0(\epsilon),x_2)$. 
Therefore for $x\geq x_3(\epsilon)$ we have
\[
\frac{f(x)}{x}a\left(F(x)/M\right)>(1-\epsilon) \frac{a\left(F(x)/M\right)}{F(x)}.
\]
By \eqref{eq.asmall} we therefore have 
\[
\lim_{x\to\infty} \frac{a\left(F(x)/M\right)}{F(x)/M}=0,
\]
and so
\[
\lim_{t\to\infty} \frac{a(t)}{t}=0.
\]
Therefore there exists $T_2>0$ such that $a(t)>0$ and $Mt-a(t)>0$ for all $t\geq T_2$. Also, since 
$Mt-a(t)\to\infty$ as $t\to\infty$, there is $T_3>0$ such that $F^{-1}(Mt-a(t))>x_1$ for all $t\geq T_3$. 
Let $T_4=\max(T_2,T_3)$.

Let $y$ be the solution of \eqref{eq.equivode} with $y(0)=1$. Then $y(t)=F^{-1}(Mt)$ for $t\geq 0$. 
Hence for $t\geq T_4$, by the mean value theorem there exists $\theta_t\in [0,1]$ such that  
\begin{align*}
F^{-1}(Mt-a(t))&=y\left(t-\frac{a(t)}{M}\right)
=y(t)+y'\left(t-\frac{\theta_t a(t)}{M}\right)\cdot \frac{-a(t)}{M} \\
&=F^{-1}(Mt)-f\left(y\left(t-\frac{\theta_t a(t)}{M}\right)\right)\cdot  a(t).
\end{align*}
Next, since $a(t)>0$ for $t\geq T_4$ and $\theta_t\in [0,1]$, we have that 
$t\geq t-\theta_t a(t)/M\geq t-a(t)/M>0$ for all $t\geq T_4$. Since $y$ is increasing, we have 
$y(t)\geq y(t-\theta_t a(t)/M)\geq y(t-a(t)/M)=F^{-1}(Mt-a(t))>x_1$ for $t\geq T_4$. Therefore we have 
\[
f\left(y\left(t-\frac{a(t)}{M}\right)\right)\leq 
f\left(y\left(t-\frac{\theta_t a(t)}{M}\right)\right)\leq f(y(t))=f(F^{-1}(Mt)).
\] 
Therefore 
\begin{align*}
F^{-1}(Mt)>F^{-1}(Mt-a(t))&\geq F^{-1}(Mt)-f(F^{-1}(Mt))a(t), \quad t\geq T_4, \\
F^{-1}(Mt-a(t))&\leq  F^{-1}(Mt)-f(F^{-1}(Mt-a(t)))a(t), \quad t\geq T_4.
\end{align*}
To finish the proof of part (a), we divide by $F^{-1}(Mt)$ across the first inequality, let $t\to\infty$ and apply \eqref{eq.asmall}. 

To prove part (b), divide the second inequality by $F^{-1}(Mt-a(t))$ and rearrange to get
\[
\frac{F^{-1}(Mt)}{F^{-1}(Mt-a(t))}-1\geq \frac{f(F^{-1}(Mt-a(t)))}{F^{-1}(Mt-a(t))}a(t)>0, \quad t\geq T_4.
\]
Letting $t\to\infty$ and using \eqref{eq.Fmin1aFmin1} we see that 
\[
\lim_{t\to\infty} \frac{f(F^{-1}(Mt-a(t)))}{F^{-1}(Mt-a(t))}a(t)=0.
\]
By Lemma~\ref{lemma.fprdown}, we have that \eqref{eq.fxdecalmost} holds. Since $F^{-1}(Mt)\sim F^{-1}(Mt-a(t))$ 
and $F^{-1}(Mt)\to\infty$ as $t\to\infty$, for every $\epsilon>0$ there is $T_1(\epsilon)>0$ such that 
$F^{-1}(Mt-a(t))>x_0(\epsilon)$ for all $t\geq T_1(\epsilon)$. Since $a(t)>0$ for all $t>T^\ast$, for $t>\max(T^\ast,T_1(\epsilon))$, we have 
\[
0<
\frac{f(F^{-1}(Mt))}{F^{-1}(Mt)}< (1+\epsilon)\frac{f(F^{-1}(Mt-a(t)))}{F^{-1}(Mt-a(t))}.
\]
Hence 
\[
\lim_{t\to\infty} \frac{f(F^{-1}(Mt))}{F^{-1}(Mt)}a(t)=0.
\]
Making the substitution $u=F^{-1}(Mt)$ gives \eqref{eq.asmall}, completing the proof of part (b).
\end{proof}
We are now in a position to prove Theorem~\ref{thm.2.2.1}.
\begin{proof}[Proof of Theorem~\ref{thm.2.2.1}]
The proof that (a) implies (b) is the subject of Theorem~\ref{thm.2.2} when $\lambda=0$. We now prove that 
(b) implies (a), with the additional hypothesis that $f'$ is decreasing on $[x_2,\infty)$. 
Without assuming the rate of growth of $f$
(i.e., absent the hypothesis that $f$ obeys \eqref{eq.lambda}), we can proceed as in the proof of Theorem~\ref{thm.2.2}
to show that 
\[
\lim_{t\to\infty} \frac{F(x(t))-Mt}{\log f(x(t))}=-C.
\]
Since $x(t)\to\infty$ as $t\to\infty$, there is $T'>0$ such that the functions  
\[
C(t):=-\frac{F(x(t))-Mt}{\log f(x(t))}, \quad a(t):=C(t)\log f(x(t)), \quad t>T',
\]
are well--defined. Moreover, granted the usual tacit assumption that $f(x)\to\infty$ as $x\to\infty$, 
we have that there is $T''>0$ such that $a(t)>0$ and $C(t)>0$ for all $t>T''$, and $C(t)\to C$ as $t\to\infty$.
By the definition of $C$ and $a$, we get
\[
x(t)=F^{-1}(Mt-a(t)), \quad t>T'.
\]  
Therefore, by part (b) of Lemma~\ref{lemma.lemmaa1}, since $x(t)\sim F^{-1}(Mt)$ by hypothesis, we have that 
\begin{equation*} 
\lim_{x\to\infty} \frac{f(x)}{x}a\left(F(x)/M\right)=0.
\end{equation*}
Now, since $x(t)\sim F^{-1}(Mt)$ as $t\to\infty$, and $f$ is ultimately increasing with ultimately decreasing derivative, 
and $f(x)\to\infty$ as $x\to\infty$, we may put $f$ in the role of $\varphi$ in Lemma~\ref{asym_equiv_copy_ap}, $x$ in the role of $b$ and $t\mapsto F^{-1}(Mt)$ in the role of $c$ to get 
\[
f(x(t))\sim f(F^{-1}(Mt)) \text{ as $t\to\infty$}. 
\]
Therefore $\log f(x(t))\sim \log f(F^{-1}(Mt))$ as $t\to\infty$ (by elementary considerations, or by identifying $\varphi=\log$ in Lemma~\ref{asym_equiv_copy_ap}, for example). Hence
\[
a(t)\sim C\log  f(F^{-1}(Mt)) \text{ as $t\to\infty$}.
\]
Since $F(x)/M\to\infty$ as $x\to\infty$, we have 
\[
a(F(x)/M) \sim C\log f(x), \quad \text{ as $x\to\infty$}.
\]
Therefore $f(x)/x \cdot \log f(x)\to 0$ as $x\to\infty$. Finally, by using the identity 
\[
\frac{f(x)}{x}\log x = -\frac{f(x)}{x} \log \left(\frac{f(x)}{x}\right) + \frac{f(x)}{x} \log f(x),
\]
(which holds for all $x$ sufficiently large) and noting that $y\log y\to 0$ as $y\to 0^+$, and $f(x)/x\to 0$ as $x\to\infty$, we see 
that $f(x)/x \cdot \log x\to 0$ as $x\to\infty$, as required.
\end{proof}

We are also in a position to prove Theorem~\ref{thm.2.2.1.1}.
%
\begin{proof}[Proof of Theorem~\ref{thm.2.2.1.1}]
Define $\varphi(x)=\log f(F^{-1}(x))$. Since $f$ is ultimately increasing and $F^{-1}$ is increasing, $\varphi$ is ultimately increasing and $\varphi(x)\to\infty$ 
as $x\to\infty$. Now $\varphi'(x)=f'(F^{-1}(x))$. Therefore, as $f'$ is ultimately decreasing, 
 $\varphi'$ is ultimately decreasing with $\varphi'(x)\downarrow 0$ as $x\to\infty$. 
As part of the proof of Theorem~\ref{thm.2.2} it was shown that the solution $x$ of \eqref{eq.fde} obeys 
$F(x(t))/t\to M$ as $t\to\infty$.
Now we apply Lemma~\ref{asym_equiv_copy_ap} with 
$b(t)=F(x(t))$, $c(t)=Mt$ and $\varphi$ as defined to get 
\[
\lim_{t\to\infty} \frac{\log f(x(t))}{\log f(F^{-1}(Mt))}=1.
\]
In the proof of Theorem~\ref{thm.2.2} it was shown that the limit
\[
\lim_{t\to\infty} \frac{F(x(t))-Mt}{\log f(x(t))}=-C
\]
holds. Furthermore, as $f(x)/(x/\log x)\to\infty$ and $f(x)/x\to 0$ as $x\to\infty$, we have that $\log f(x)/\log x\to 1$ 
as $x\to\infty$, so taking these limits together, we arrive at 
\[
\lim_{t\to\infty} -\frac{F(x(t))-Mt}{\log F^{-1}(Mt)}=C.
\] 
Finally the function $c:[1,\infty) \to\mathbb{R}$ given by   
\[
c(t):=-\frac{F(x(t))-Mt}{\log F^{-1}(Mt)}, \quad t\geq 1
\]
is well--defined, in $C^1$, and obeys $c(t)\to C$ as $t\to\infty$. Rearranging this identity in terms of $x$  yields 
the result. 
\end{proof}

In addition to Lemma~\ref{lemma.lemmaa1}, we will need one more preparatory result in order to 
prove Theorem~\ref{thm.2.4}: we state and prove it now. 
\begin{lemma} \label{lemma.la2}
Let $M>0$. Suppose that $f(x)>0$ for all $x>0$, $f'(x)>0$ for all $x > x_1$, $f'(x) \to 0$ as $x \to \infty$. Define $F$ as in \eqref{def.F}. Suppose that $\epsilon$ is a positive, non--decreasing and measurable function with $\epsilon(t)\to 0$ as $t\to\infty$. Then
\begin{equation} \label{eq.epsbig}
\lim_{x\to\infty} \frac{f(x)}{x}\int_{0}^{F(x)/M} \epsilon(s)\,ds=+\infty
\end{equation}
implies
\begin{equation*}
\lim_{t\to\infty} \frac{F^{-1}(Mt-\int_0^t \epsilon(s)\,ds)}{F^{-1}(Mt)}=0.
\end{equation*}
\end{lemma}
\begin{proof}
Let $y$ be the solution of \eqref{eq.equivode} with $y(0)=1$. Then $y(t)=F^{-1}(Mt)$ for $t\geq 0$.
Define 
\[
K(t):=\frac{1}{M}\int_0^t \epsilon(s)\,ds, \quad \kappa(t)=t-K(t).
\]
Then $K$ is non--decreasing. Also as $\epsilon(t)\to 0$ as $t\to\infty$, we have $K(t)/t\to 0$ 
as $t\to\infty$. Hence $\kappa(t)\to\infty$ as $t\to\infty$ and indeed $\kappa(t)/t\to 1$ as $t\to\infty$. 
Since $\epsilon(t)\to 0$ as $t\to\infty$, it follows that $0\leq \epsilon(t)<M/8$ for all $t\geq T_1$. Thus 
for $t>s\geq T_1$, we have 
\[
\kappa(t)-\kappa(s) = t-s -\frac{1}{M}\int_s^t \epsilon(u)\,ds \geq  \frac{7}{8}(t-s).
\]   
Hence $\kappa$ is increasing on $[T_1,\infty)$, so $\kappa^{-1}$ is well--defined and $\kappa^{-1}(t)/t\to 1$ 
as $t\to\infty$. Also, as $\kappa(t)<t$ for all $t$ sufficiently large we have $\kappa^{-1}(t)>t$ for 
all $t$ sufficiently large (say $t\geq T_2$). Thus for $t\geq T_2$, as $\epsilon$ is non--increasing, we have 
\[
0\leq K(\kappa^{-1}(t))-K(t)=\int_t^{\kappa^{-1}(t)} \epsilon(s)\,ds \leq \epsilon(t)(\kappa^{-1}(t)-t).
\]
By the definition of $\kappa$, there is $T_3>0$ such that $t=\kappa^{-1}(t)-K(\kappa^{-1}(t))$ for $t\geq T_3$. Also, there is $T_4>0$ such that $\epsilon(t)<1$ for all $t\geq T_4$. Thus for $t\geq T_5=\max(T_2,T_3,T_4)$ we have 
\[
0\leq K(\kappa^{-1}(t))-K(t) \leq \epsilon(t)K(\kappa^{-1}(t)).
\]   
and indeed 
\[
1\leq \frac{K(\kappa^{-1}(t))}{K(t)}\leq \frac{1}{1-\epsilon(t)}, \quad t\geq T_5.
\]
Therefore 
\begin{equation} \label{eq.Kasykappa}
\lim_{t\to\infty} \frac{K(\kappa^{-1}(t))}{K(t)}=1.
\end{equation}

Since $\kappa(t)\to\infty$ as $t\to\infty$, and $\kappa$ is increasing, we have 
\begin{align*}
\lim_{t\to\infty} \frac{f(y(t-K(t)))}{y(t-K(t))}K(t)
&=\lim_{t\to\infty} \frac{f(y(\kappa(t)))}{y(\kappa(t))}K(t)=\lim_{z\to\infty} \frac{f(y(z))}{y(z)}K(\kappa^{-1}(z))\\
&=\lim_{z\to\infty} \frac{f(y(z))}{y(z)}K(z) \cdot \frac{K(\kappa^{-1}(z))}{K(z)}\\
&=\lim_{z\to\infty} \frac{f(F^{-1}(Mz))}{F^{-1}(Mz)}K(z) \cdot \frac{K(\kappa^{-1}(z))}{K(z)}\\
&=+\infty
\end{align*}
where we have used  \eqref{eq.epsbig} and \eqref{eq.Kasykappa} at the last step. 
Hence 
\begin{equation} \label{eq.fyasy}
\lim_{t\to\infty} \frac{f(y(t-K(t)))}{y(t-K(t))}K(t)=+\infty.
\end{equation}

By hypothesis, there is $T_6>0$ such that $t-K(t)>0$ for all $t\geq T_6$ and also that 
$F^{-1}(t-K(t))>x_1$ for all $t\geq T_7$. Let $T_8=\max(T_6,T_7)$. Then 
Next, for $t\geq T_8$ we have 
\[
F^{-1}\left(Mt-\int_0^t \epsilon(s)\,ds\right)=y(t-K(t)),
\]
so by the mean value theorem, there is $\theta_t\in [0,1]$ such that 
\[
y(t-K(t))=y(t)-y'(t-\theta_t K(t))K(t)=y(t)-Mf(y(t-\theta_t K(t)))K(t).
\]
Since $K(t)\geq 0$, $t\geq T_8$, and $\theta_t\in [0,1]$, $t-\theta_t K(t)\geq t-K(t)>0$. 
Since $y$ is increasing and $t\geq T_8$, $y(t-\theta_t K(t))\geq y(t-K(t))=F^{-1}(t-K(t))>x_1$. 
Therefore, as $f$ is increasing on $[x_1,\infty)$, we have 
\[
f(y(t-\theta_t K(t))) \geq f(y(t-K(t))).
\]
Hence for $t\geq T_8$
\[
y(t-K(t))=y(t)-Mf(y(t-\theta_t K(t)))K(t)\leq y(t)-Mf(y(t-K(t)))K(t).
\]
Therefore
\begin{equation*} 
y(t-K(t))+Mf(y(t-K(t)))K(t)\leq y(t), \quad t\geq T_8,
\end{equation*}
and so 
\[
\frac{y(t)}{y(t-K(t))}\geq 1 + M\frac{f(y(t-K(t)))}{y(t-K(t))}\cdot K(t), \quad t\geq T_8.
\]
Hence by \eqref{eq.fyasy} we see that 
\begin{equation} \label{eq.ykyt}
\lim_{t\to\infty} \frac{y(t)}{y(t-K(t))} = +\infty. 
\end{equation}
Finally,  since $F^{-1}\left(Mt-\int_0^t \epsilon(s)\,ds\right)=y(t-K(t))$, we see from \eqref{eq.ykyt} 
that 
\[
\lim_{t\to\infty} \frac{F^{-1}\left(Mt-\int_0^t \epsilon(s)\,ds\right)}{F^{-1}(Mt)}
=\lim_{t\to\infty} \frac{y(t-K(t))}{y(t)}
=0,
\]
completing the proof.
\end{proof}
We are now in a position to prove Theorem~\ref{thm.2.4}. 
\begin{proof}[Proof of Theorem~\ref{thm.2.4}]
As before we have defined  $\epsilon_1(t)=\int_{(t,\infty)}\mu(ds)$ for $t\geq 0$ and 
\[
\delta_1(t)=\epsilon_1(t)f(x(t)), \quad t\geq 0.
\]
Clearly $\delta_1(t)>0$ for all $t\geq 0$. Define also $\delta_2$ by 
\[
\delta_2(t)=\int_{[0,t]} \mu(ds)\left(f(x(t))-f(x(t-s))\right), \quad t\geq 0.
\]
We get 
\[
\delta_2(t)=\int_{0}^t M(t-u,t) f'(x(u))x'(u)\,du.
\]
Then as $f$ is increasing on $[0,\infty)$, we have that $\delta_2(t)>0$ for all $t> 0$ 
and hence
\begin{equation}\label{eq.xrep0}
x'(t)=Mf(x(t))-\delta_1(t)-\delta_2(t), \quad t\geq 0.
\end{equation}
Next if we define  
\[
\tilde{I}_2(t)=\frac{\delta_2(t)}{Mf(x(t))}, \quad t\geq 0,
\]
integration of \eqref{eq.xrep0} yields
\begin{equation} \label{eq.xrep1} 
F(x(t))-Mt=F(x(0))- \int_{0}^t \epsilon_1(s)\,ds -\int_0^t M\tilde{I}_2(s)\,ds, \quad t\geq 0.
\end{equation}

We now prove part (i) of the Theorem. To start with, we get an upper estimate for $x$.
Since $x'(t)\leq Mf(x(t))$ for all $t\geq 0$, we have 
\[
\tilde{I}_2(t)\leq \int_{0}^t M(t-u,t)f'(x(u))\,du.
\]
Hence 
\begin{align*}
\int_0^t M\tilde{I}_2(s)\,ds 
&\leq M\int_0^t \int_{0}^s M(s-u,s)f'(x(u))\,du \,ds \\
&= \int_0^t \int_u^t M(s-u,s)\,ds f'(x(u))\,du.
\end{align*}
Now for $t\geq u$ we have
\begin{align*}
\int_u^t M(s-u,s)\,ds 
&= \int_u^t \int_{[s-u,s]} \mu(dv) \,ds
=\int_0^{t-u} \int_{[w,w+u]} \mu(dv) \,dw 
=M_1(u,t), 
\end{align*}
by the definition of $M_1$ in \eqref{def.M1}. 
Now, from \eqref{eq.estM1}, we have 
\[
M_1(u,t)\leq  \int_{[0,t]} v  \mu(dv), \quad t\geq u.
\]
Therefore
\[
\int_u^t M(s-u,s)\,ds  \leq \int_{[0,t]} v  \mu(dv), \quad t\geq u.
\]
Therefore
\begin{align*}
\int_0^t M\tilde{I}_2(s)\,ds 
\leq M\int_0^t \int_{0}^s M(s-u,s)f'(x(u))\,du \,ds = \int_{[0,t]} v  \mu(dv) \int_0^t  Mf'(x(u))\,du.
\end{align*}
Hence we have from \eqref{eq.xrep1} that
\[
F(x(t))-Mt\geq F(x(0))- \int_{0}^t \epsilon_1(s)\,ds - \int_{[0,t]} v  \mu(dv) \int_0^t  Mf'(x(u))\,du, \quad t\geq 0.
\]
Next, we have that $x'(t)>M(1-\epsilon)f(x(t))$ for all $t\geq T_1(\epsilon)$. Define 
\[
K_1(\epsilon):=\int_0^{T_1(\epsilon)} Mf'(x(u))\,du.
\]
Therefore for $t\geq T_1(\epsilon)$ 
we have 
\begin{align*}
\int_0^t  Mf'(x(u))\,du 
&=  K_1(\epsilon) + \int_{T_1}^t \frac{f'(x(u))}{f(x(u))}\cdot\frac{Mf(x(u))}{x'(u)}x'(u)\,du\\
&\leq K_1(\epsilon) + \frac{1}{1-\epsilon}\int_{T_1}^t \frac{f'(x(u))}{f(x(u))}x'(u)\,du\\
&= K_1(\epsilon) + \frac{1}{1-\epsilon} \left( \log f(x(t))  - \log f(x(T_1))\right).
\end{align*}
Since $f(x)/x\to 0$ as $x\to\infty$, for every $\epsilon>0$ there is $T_2(\epsilon)>0$ and $K_2(\epsilon)>0$ 
such that 
\[
\int_0^t  Mf'(x(u))\,du  \leq K_2(\epsilon) + \frac{1}{1-\epsilon} \log x(t), \quad t\geq T_2(\epsilon).
\]
Next, we have $x(t)\leq F^{-1}(F(x(0))+Mt)$ for $t\geq 0$, so 
\[
\int_0^t  Mf'(x(u))\,du  \leq K_2(\epsilon) + \frac{1}{1-\epsilon} \log F^{-1}(F(x(0))+Mt), 
\quad t\geq T_2(\epsilon).
\] 
Finally, as $F^{-1}(c+Mt)\sim F^{-1}(Mt)$ as $t\to\infty$, we have 
\[
\limsup_{t\to\infty} \frac{\int_0^t  Mf'(x(u))\,du}{\log  F^{-1}(Mt)}\leq 1,
\]
and moreover 
\begin{equation} \label{eq.xestupper1}
\limsup_{t\to\infty} \frac{x(t)}{F^{-1}(Mt)}\leq 1.
\end{equation}
Hence for every $\epsilon\in (0,1)$ there is $T_3(\epsilon)>0$ such that $t\geq T_3(\epsilon)$ implies  
\[
F(x(t))-Mt\geq F(x(0))- \int_{0}^t \int_{[s,\infty)}\mu(du) \,ds - \int_{[0,t]} s  \mu(ds) (1+\epsilon) \log F^{-1}(Mt).
\]
Define 
\[
a_2(t)= -F(x(0)) + \int_{0}^t \int_{[s,\infty)}\mu(du) \,ds + \int_{[0,t]} s  \mu(ds) (1+\epsilon) \log F^{-1}(Mt).
\]
Then 
\[
\liminf_{t\to\infty} \frac{x(t)}{F^{-1}(Mt)}
\geq \liminf_{t\to\infty} \frac{F^{-1}(Mt-a_2(t))}{F^{-1}(Mt)}.
\]
Clearly $a_2(t)>0$ for all $t$ sufficiently large. Finally
\begin{multline*}
\frac{f(x)}{x}
a_2(F(x)/M) = \frac{f(x)}{x}\int_{0}^{F(x)/m} \int_{[s,\infty)}\mu(du) \,ds 
\\+ \int_{[0,F(x)/M]} s  \mu(ds) 
(1+\epsilon) \frac{f(x)}{x}\log x - \frac{f(x)}{x}F(x(0)),
\end{multline*}
so by  \eqref{eq.suff11} and \eqref{eq.suff21}, we have $a_2(F(x)/M)f(x)/x
\to 0$ as $x\to\infty$. Hence with $a_2$ in the role of $a$ in Lemma~\ref{lemma.lemmaa1}, we have 
\[
\lim_{t\to\infty}  \frac{F^{-1}(Mt-a_2(t))}{F^{-1}(Mt)} =1,
\]
and so 
\[
\liminf_{t\to\infty} \frac{x(t)}{F^{-1}(Mt)} \geq 1.
\]
Combining this with \eqref{eq.xestupper1} proves part (i). 

We now prove part (ii). From \eqref{eq.xrep0}, and the fact that $\delta_2(t)>0$ we have 
\[
x'(t)\leq Mf(x(t))-\delta_1(t)=Mf(x(t))-\epsilon_1(t)f(x(t)), \quad t\geq 0.
\]
Dividing by $f(x(t))$ and integrating gives
\begin{equation} \label{eq.xrep3}
x(t)\leq F^{-1}\left(F(x(0))+ Mt - \int_0^t \epsilon_1(s)\,ds \right),\quad t\geq 0.
\end{equation}
Since $\epsilon_1(t)=\int_{[t,\infty)} \mu(ds)$, we see that $\epsilon_1$ is positive, non--increasing and 
obeys $\epsilon_1(t)\to 0$ as $t\to\infty$. 
Therefore by \eqref{eq.suff10}
\begin{equation*} 
\lim_{x\to\infty} \frac{f(x)}{x}\int_{0}^{F(x)/M} \epsilon_1(s)\,ds
=\lim_{x\to\infty}\frac{f(x)}{x}\int_{[0,F(x)/M]} \int_{[s,\infty)}  \mu(du)\,ds
=+\infty
\end{equation*}
so the condition \eqref{eq.epsbig} in Lemma~\ref{lemma.la2} holds. Therefore as all conditions of Lemma~\ref{lemma.la2}
hold with $\epsilon_1$ in the role of $\epsilon$, we get
\[
\lim_{t\to\infty} \frac{F^{-1}\left(Mt - \int_0^t \epsilon_1(s)\,ds \right)}{F^{-1}(Mt)}=0.
\]
Finally, as $F^{-1}(c+Mt)\sim F^{-1}(Mt)$ as $t\to\infty$ for any $c\in\mathbb{R}$, it follows from 
\eqref{eq.xrep3} that 
\[
\limsup_{t\to\infty} \frac{x(t)}{F^{-1}(Mt)}
\leq 
\limsup_{t\to\infty} \frac{F^{-1}\left(Mt - \int_0^t \epsilon_1(s)\,ds \right)}{F^{-1}(Mt)}
=0,
\] 
and hence part (ii) of Theorem~\ref{thm.2.4} has been proven.

To prove part (iii), we revisit  \eqref{eq.xrep3}, namely 
\[
x(t)\leq F^{-1}\left(F(x(0))+ Mt - \int_0^t \epsilon_1(s)\,ds \right),\quad t\geq 0. 
\]
Since $F^{-1}(c+Mt)\sim F^{-1}(Mt)$ as $t\to\infty$ for any $c\in\mathbb{R}$, and by hypothesis we have 
$x(t)\sim F^{-1}(Mt)$ as $t\to\infty$, we see that 
\[
\liminf_{t\to\infty} \frac{F^{-1}\left(Mt - \int_0^t \epsilon_1(s)\,ds \right)}{F^{-1}(Mt)}\geq 1.
\] 
On the other hand, as $\epsilon_1(t)>0$ for all $t\geq 0$, we have the trivial limit
\[
\limsup_{t\to\infty} \frac{F^{-1}\left(Mt - \int_0^t \epsilon_1(s)\,ds \right)}{F^{-1}(Mt)}\leq 1,
\]
and so
\[
\lim_{t\to\infty} \frac{F^{-1}\left(Mt - \int_0^t \epsilon_1(s)\,ds \right)}{F^{-1}(Mt)}= 1.
\]
Now set $a(t)=\int_0^t \epsilon_1(s)\,ds$. Since $f'$ is decreasing on $[x_2,\infty)$, by Lemma~\ref{lemma.lemmaa1} part (b) it follows that 
\[
\lim_{x\to\infty} \frac{f(x)}{x}a\left(F(x)/M\right)=0, 
\]
which is precisely \eqref{eq.suff21}. This completes the proof of part (iii).
\end{proof}

\section{Proof of Theorem~\ref{thm.2.5}}
We start with the proof of a preliminary result. 

\begin{lemma} \label{lemma.asyeqivMints}
Suppose that $M$ is defined by 
\[
M(t-u,t)=\int_{[t-u,t]} \mu(ds), \quad t\geq u\geq 0,
\]
where $\mu\in M([0,\infty);\mathbb{R}^+)$. Suppose that $b$ and $c$ are continuous functions with 
$b(t)\sim c(t)$ as $t\to\infty$ with $0<c(t)\to\infty$ as 
$t\to\infty$. Then 
\[
\int_{0}^t M(t-u,t)c(u)\,du \to
\infty \text{ as $t\to\infty$}, 
\]
and 
\[
\int_{0}^t M(t-u,t)b(u)\,du  \sim \int_{0}^t M(t-u,t)c(u)\,du \text{ as $t\to\infty$}.
\]
\end{lemma}
\begin{proof}
Define 
\[
\delta(t)=\int_{0}^t M(t-u,t)b(u)\,du, \quad \tilde{\delta}(t)=\int_{0}^t M(t-u,t)c(u)\,du, \quad t\geq 0.
\]
Now for $t\geq 1$ we have 
\[
\tilde{\delta}(t)=\int_{0}^t M(t-u,t)c(u)\,du\geq \int_{t-1}^t M(t-u,t)\,du \cdot \inf_{u\in [t-1,t]}c(u).
\]
Since 
\[
\int_{t-1}^t M(t-u,t)\,du = \int_{t-1}^t \int_{[t-u,t]} \mu(ds)\,du
=\int_{0}^1 \int_{[v,t]} \mu(ds)\,dv =\int_{[0,t]} \{1\wedge s\} \mu(ds), 
\]
the positivity of $\mu$ implies that $\tilde{\delta}(t)\to\infty$ as $t\to\infty$.

Since $b(t)\sim c(t)$ it follows for every $\epsilon\in (0,1)$ there is $T_1(\epsilon)>0$ such that 
$(1-\epsilon)c(t)<b(t)<(1+\epsilon)c(t)$ for all $t\geq T_1(\epsilon)$. Hence for $t\geq T_1(\epsilon)$ we have
by the positivity of $b$ and $c$ on $[T_1,\infty)$ and $M$ on its domain that 
\[
(1-\epsilon)\int_{T_1}^t M(t-u,t)c(u)\,du \leq \int_{T_1}^t M(t-u,t)b(u)\,du \leq 
(1+\epsilon)\int_{T_1}^t M(t-u,t)c(u)\,du 
\]
Now 
\[
\int_{T_1}^t M(t-u,t)c(u)\,du = \tilde{\delta}(t) - \int_0^{T_1} M(t-u,t)c(u)\,du
\]
Thus, as $M(t-u,t)\leq M$ and $c$ is non--negative, we have for $t\geq T_1$ 
with $C(\epsilon):=M\int_0^{T_1} c(u)\,du$, 
\[
\int_{T_1}^t M(t-u,t)c(u)\,du \leq \tilde{\delta}(t), \quad
\int_{T_1}^t M(t-u,t)c(u)\,du \geq \tilde{\delta}(t) - C(\epsilon).
\]
Hence for $t\geq T_1$ we have 
\[
(1-\epsilon)\left(\tilde{\delta}(t) - C(\epsilon)  \right) \leq \int_{T_1}^t M(t-u,t)b(u)\,du \leq 
(1+\epsilon)\tilde{\delta}(t).
\]
Thus by the definition of $\delta$ we have for $t\geq T_1(\epsilon)$ 
\[
(1-\epsilon)\left(\tilde{\delta}(t) - C(\epsilon)  \right)
\leq 
 \delta(t)-\int_0^{T_1} M(t-u,t)b(u)\,du
\leq 
(1+\epsilon)\tilde{\delta}(t).
\]
Define 
\[
B(\epsilon):= M\int_0^{T_1} |b(u)|\,du<+\infty.
\]
Then 
\[
\left|
\int_0^{T_1} M(t-u,t)b(u)\,du \right|\leq B(\epsilon),
\]
and so for $t\geq T_1(\epsilon)$ 
\[
(1-\epsilon)\left(\tilde{\delta}(t) - C(\epsilon)  \right) -B(\epsilon)
\leq  \delta(t)\leq 
(1+\epsilon)\tilde{\delta}(t) + B(\epsilon).
\]
Dividing by $\tilde{\delta}(t)$, letting $t\to\infty$ and remembering that $\tilde{\delta}(t)\to\infty$ as $t\to\infty$
we get 
\[
1-\epsilon \leq \liminf_{t\to\infty} \frac{\delta(t)}{\tilde{\delta}(t)}\leq 
\limsup_{t\to\infty} \frac{\delta(t)}{\tilde{\delta}(t)}\leq 1+\epsilon.
\]
Letting $\epsilon\to 0^+$ completes the proof.
\end{proof}

\begin{proof}[Proof of Theorem~\ref{thm.2.5}]
We note that $\delta_2$ is given by  
\[
\delta_2(t)=\int_{0}^t M(t-u,t) f'(x(u))x'(u)\,du, \quad t\geq 0.
\]
Define also 
\[
\epsilon_2(t)=\frac{1}{f(x(t))}\delta_2(t), \quad t\geq 0.
\]
Then as $f$ is increasing on $[0,\infty)$, we have that $\epsilon_2(t)>0$ for all $t> 0$ 
and we have from \eqref{eq.xrep0} that 
\begin{equation*} 
x'(t)=Mf(x(t))-\epsilon_1(t)f(x(t))-\epsilon_2(t)f(x(t)), \quad t\geq 0.
\end{equation*}
Dividing by $f(x(t))$ and integrating yields
\[
x(t)=F^{-1}\left(F(x(0))+Mt-a(t)\right), \quad t\geq 0,
\]
where 
\[
a(t):=a_1(t)+a_2(t)= \int_0^t \epsilon_1(s)\,ds + \int_0^t \epsilon_2(s)\,ds, \quad t\geq 0.
\]
Since $x(t)\sim F^{-1}(Mt)$ as $t\to\infty$ and $f'$ is decreasing, we can replicate the proof of part (iii)
 of Theorem~\ref{thm.2.4} to get
\[
\lim_{x\to\infty} \frac{f(x)}{x}a\left(\frac{F(x)}{M}\right)=0.
\]  
Since both $a_1$ and $a_2$ are positive, this implies 
\[
\lim_{x\to\infty} \frac{f(x)}{x}a_1\left(\frac{F(x)}{M}\right)=0, \quad
\lim_{x\to\infty} \frac{f(x)}{x}a_2\left(\frac{F(x)}{M}\right)=0.
\]
The first condition is nothing but \eqref{eq.suff21}. 

We now determine the asymptotic behaviour of $a_2(t)$ as $t\to\infty$, and show that the second limit implies 
\eqref{eq.suff31}. First, because $x'(t)\sim Mf(x(t))$ 
as $t\to\infty$ and $f'(x)\sim f(x)/x$ as $x\to\infty$ 
we have that 
\[
b(t):=f'(x(t))x'(t)\sim f'(x(t))Mf(x(t)) \sim M \frac{f^2(x(t))}{x(t)}=:c_1(t) \text{ as $t\to\infty$}.
\]
Set $g(x):=f^2(x)/x$: then $g\in \text{RV}_\infty(1)$. Therefore as $x(t)\sim F^{-1}(Mt)$ 
as $t\to\infty$, we have that $g(x(t))\sim g(F^{-1}(Mt))$ as $t\to\infty$. Hence 
\[
b(t)\sim c_1(t)=M g(x(t)) = M\frac{f^2(F^{-1}(Mt))}{F^{-1}(Mt)}=:c(t) \text{ as $t\to\infty$}.
\] 
Moreover, $c(t)\to\infty$ as $t\to\infty$. Thus by Lemma~\ref{lemma.asyeqivMints}, we have that 
\[
\delta_2(t)=\int_0^t M(t-u,t)b(u)\,du \sim \int_0^t M(t-u,t)c(u)\,du=:\tilde{\delta}(t) \text{ as $t\to\infty$}.
\]
Since $x(t)\sim F^{-1}(Mt)$ as $t\to\infty$, we have $f(x(t))\sim f(F^{-1}(Mt))$ as $t\to\infty$. 
Therefore, by the definition of $\epsilon_2$, $\tilde{\delta}$ and $c$ we have
\[
\epsilon_2(t)\sim 
\frac{M}{f(F^{-1}(Mt))} \int_0^t M(t-u,t) \frac{f^2(F^{-1}(Mu))}{F^{-1}(Mu)}\,du=:\tilde{\epsilon}_2(t), 
\quad \text{ as $t\to\infty$}.
\] 
By making the substitution $v=F^{-1}(Mt)$ and $w=F^{-1}(Mu)$ in the iterated integral, 
for any $T>0$ we have from \eqref{def.K} that 
\begin{align*}
\int_0^T \tilde{\epsilon}_2(t)\,dt &= M\int_0^T \int_0^t 
\frac{M}{f(F^{-1}(Mt))} \int_0^t M(t-u,t) \frac{f^2(F^{-1}(Mu))}{F^{-1}(Mu)}\,du\,dt\\
&=M\int_1^{F^{-1}(MT)} K(v)\frac{1}{f^2(v)}\,dv.
\end{align*}

If $x\mapsto \int_1^{x} K(v)\frac{1}{f^2(v)}\,dv$ tends to a finite limit, 
then we have 
\[
\lim_{x\to\infty} \frac{f(x)}{x} \int_1^{x} K(v)\frac{1}{f^2(v)}\,dv=0,
\]
which gives \eqref{eq.suff31} directly. 

On the other hand, if $x\mapsto \int_1^{x} K(v)\frac{1}{f^2(v)}\,dv$ tends to $+\infty$ as $x\to\infty$, 
because
\[
M\int_1^{x} K(v)\frac{1}{f^2(v)}\,dv=
\int_0^{F(x)/M} \tilde{\epsilon}_2(t)\,dt 
\]
we have that $\tilde{\epsilon}_2(t)\to\infty$ as $t\to\infty$. Since $\epsilon_2(t)\sim \tilde{\epsilon}_2(t)$, we have that 
\[
a_2(t)=\int_0^t \epsilon_2(s)\,ds \sim \int_0^t \tilde{\epsilon}_2(s)\,ds \to\infty \text{ as $t\to\infty$}.
\]
Thus 
\[
a_2(F(x)/M) \sim \int_0^{F(x)/M} \tilde{\epsilon}_2(s)\,ds = M\int_1^{x} K(v)\frac{1}{f^2(v)}\,dv
\quad\text{ as $x\to\infty$}.
\]
Therefore, as 
\[
\lim_{x\to\infty} \frac{f(x)}{x}a_2(F(x)/M)=0,
\]
we have 
\[
\lim_{x\to\infty} \frac{f(x)}{x}\int_1^{x} K(v)\frac{1}{f^2(v)}\,dv=0,
\]
which is  \eqref{eq.suff31}, as required.
\end{proof}

\bibliography{references}
\bibliographystyle{abbrv}

\end{document}